\documentclass{article}

\newif\ifsimax
\simaxfalse

\usepackage{amsmath,amssymb,amsthm,bm}
\usepackage[a4paper]{geometry}
\usepackage{algorithm,algorithmic}
\usepackage{array}
\usepackage{caption}
\usepackage{graphicx}
\usepackage[colorlinks=false,pdfborder={0 0 0}]{hyperref}
\usepackage{xcolor}
\usepackage{authblk}
\usepackage{multirow}
\usepackage{nicematrix}

\DeclareMathOperator*{\argmin}{\arg\min}

\DeclareMathOperator{\diag}{diag}

\DeclareMathOperator{\Span}{span}
\DeclareMathOperator{\trace}{trace}

\newcommand*{\set}[1]{\lbrace#1\rbrace}

\newcommand*{\herm}{^*}

\newcommand*{\pinv}{^{\dagger}}

\newcommand{\bmat}[1]{\begin{bmatrix}#1\end{bmatrix}}

\def\adots{\mathinner{\mkern2mu\raise1pt\hbox{.}\mkern2mu
    \raise4pt\hbox{.}\mkern2mu\raise7pt\hbox{.}\mkern1mu}}

\DeclareMathOperator{\nnz}{nnz}

\DeclareMathOperator{\spec}{spec}

\newcommand*{\nev}{{n_{\mathrm{ev}}}}
\newcommand*{\nex}{{n_{\mathrm{ex}}}}
\newcommand*{\nes}{{n_{\mathrm{es}}}}
\newcommand*{\nnow}{{n_{\mathrm{now}}}}

\newcommand*{\jwarm}{j_{\mathrm{warm}}}
\newcommand*{\rwarm}{r_{\mathrm{warm}}}

\newcommand*{\js}{{j_{\mathrm{s}}}}
\newcommand*{\je}{{j_{\mathrm{e}}}}
\newcommand*{\jp}{{j_{\mathrm{p}}}}
\newcommand*{\jl}{{j_{\mathrm{l}}}}

\newcommand*{\teta}{{\tilde\eta}}
\newcommand*{\heta}{{\hat\eta}}
\newcommand*{\ceta}{{\check\eta}}

\newcommand*{\Xexp}{{X_{\mathrm{exp}}}}
\newcommand*{\Xone}{X_k}
\newcommand*{\Xtwo}{X_{l\backslash k}}
\newcommand*{\Xthr}{{X_l^{\perp}}}

\newcommand*{\Deltaone}{\Delta_k}
\newcommand*{\Deltatwo}{\Delta_{l\backslash k}}
\newcommand*{\Deltathr}{{\Delta_l^{\perp}}}

\newcommand*{\Sigmaone}{\Sigma_k}
\newcommand*{\Sigmatwo}{\Sigma_{l\backslash k}}
\newcommand*{\Sigmathr}{{\Sigma_l^{\perp}}}

\newcommand*{\Xnew}{{\breve X}}
\newcommand*{\Xmid}{X_{\mathrm{mid}}}

\newcommand*{\epsexp}{\epsilon_{\mathrm{exp}}}

\newcommand*{\cmax}{{c_{\max}}}
\newcommand*{\Xdrop}{X_{\mathrm{drop}}}
\newcommand*{\Sdrop}{X_{\mathrm{drop}}}

\newcommand*{\termgk}{{\xi}}

\newcommand{\RMIS}[2]{#2}
\newcommand{\RMISBLOCK}[2]{#2}
\newcommand{\ISALL}{}
\newcommand{\ISALLEND}{}

\definecolor{bkgndcolor}{rgb}{1,1,1}

\graphicspath{{fig/}}

\newenvironment{keywords}{\medskip\textbf{Keywords:}}{}

\newenvironment{MSCcodes}{\medskip\textbf{AMS subject classifications (2020).}}{}

\newtheorem{theorem}{Theorem}

\newtheorem{proposition}[theorem]{Proposition}

\newtheorem{remark}{Remark}

\title{On a shrink-and-expand technique for symmetric block eigensolvers}
\author[1]{Yuqi Liu}
\author[2]{Yuxin Ma}
\author[3,4]{Meiyue Shao}
\affil[1]{School of Mathematical Sciences, Fudan University, Shanghai 200433,
China}
\affil[2]{Department of Numerical Mathematics, Faculty of Mathematics and Physics, Charles University, Sokolovsk\'{a} 49/83, 186 75 Praha 8, Czechia}
\affil[3]{School of Data Science, Fudan University, Shanghai 200433, China}
\affil[4]{MOE Key Laboratory for Computational Physical Sciences, Fudan University, Shanghai 200433, China}
\date{\today}

\begin{document}
\pagecolor{bkgndcolor}

\maketitle

\begin{abstract}
In \RMIS{}{symmetric} block eigenvalue algorithms, such as the
subspace iteration algorithm and the locally optimal block preconditioned
conjugate gradient (LOBPCG) algorithm, a large block size is often employed
to achieve robustness and rapid convergence.
However, using a large block size also increases the computational cost.
Traditionally, the block size is typically reduced after convergence of some
eigenpairs, known as deflation.
In this work, we propose a non-deflation-based, more aggressive technique,
where the block size is adjusted dynamically during the algorithm.
This technique can be applied to a wide range of block eigensolvers, reducing
computational cost without compromising convergence speed.
We present three adaptive strategies for adjusting the block size, and apply
them to four well-known eigensolvers as examples.
\RMIS{Theoretical}{Detailed theoretical} analysis and numerical experiments are provided to illustrate the
efficiency of the proposed technique.
In practice, an overall acceleration of \(20\%\) to \(30\%\) is observed.
\end{abstract}



\section{Introduction}
\label{sec:intro}

Given a large, sparse Hermitian matrix \(A\in\mathbb{C}^{n\times n}\), this
work considers to compute \RMIS{a subset of eigenparis}{\(\nev\) eigenpairs}
of \(A\) satisfying
\[
AV=V\Lambda,
\]
where the diagonal matrix \(\Lambda\in\mathbb{C}^{\nev\times\nev}\) contains
the eigenvalues, and the columns of \(V\in\mathbb{C}^{n\times\nev}\) consist
of the corresponding eigenvectors.
This problem is frequently encountered in various applications, such as PDE,
electronic structure
calculations, and machine learning; see, for example,
\cite{BGHRCV2010, KMS2023, Saad2011}.

\RMIS{
Researchers have proposed many algorithms for solving this problem, such as
the subspace iteration (SI) algorithm~\cite{Demmel1997}, preconditioned
inverse iteration (PINVIT) algorithm~\cite{N2002}, the locally optimal block
preconditioned conjugate gradient (LOBPCG) algorithm~\cite{Knyazev2001}, the
steepest descent (SD) algorithm~\cite{N2012}, the conjugate gradient (CG)
algorithm~\cite{FO1996}, the trace minimization (TraceMIN)
algorithm~\cite{SW1982}, and the Rayleigh--Ritz type method with contour
integration (CIRR)~\cite{SS2007}, etc.
These algorithms mentioned above share a common feature---a search subspace
with fixed dimension is stored and used for generating the approximation to
the desired invariant subspace in each iteration.
The dimension of the search subspace is not changed until some eigenpairs have
converged.
}{}
\RMIS{}
{To address large, sparse eigenvalue problems, projection methods are 
frequently employed.
These include eigensolvers based on the Lanczos\slash{}Arnoldi 
process~\cite{GU1977,Sorensen1997,Stewart2002}, the Davidson-type 
algorithm~\cite{SV2000,zhou2010,ZS2007}, the subspace iteration (SI) 
algorithm~\cite{Demmel1997}, the preconditioned inverse iteration (PINVIT) 
algorithm~\cite{N2002}, the locally optimal block preconditioned conjugate 
gradient (LOBPCG) algorithm~\cite{Knyazev2001}, the steepest descent (SD) 
algorithm~\cite{N2012}, the trace minimization (TraceMIN) algorithm~\cite{SW1982},
and the Rayleigh--Ritz type method with contour integration (CIRR)~\cite{SS2007}.
These techniques begin by forming a search space with a significantly reduced 
dimension compared to \(n\), and then solve the projected problem on this 
smaller subspace to acquire approximate eigenpairs.
It is important to note that nearly all these algorithms have block versions. 
The distinction among these projection methods lies in the fact that, aside 
from the Lanczos\slash{}Arnoldi-based and Davidson-type algorithms, the 
others maintain a fixed-size subspace while iteratively refining it to better 
approximate the target invariant subspace, which is the main focus of this work.
Hereafter, the term \textit{block eigensolver} primarily refers to these 
\textit{fixed block size eigensolvers}.
}

In this paper, we will generalize the \emph{shrink-and-expand} technique in
the recent work~\cite{LSS2024} to accelerate the block eigensolvers mentioned
above.
Let us first briefly introduce this technique using the SI algorithm,
which is one of the simplest and representative block eigensolver.%
\footnote{We always assume that a Rayleigh--Ritz process is applied to the SI
algorithm in this paper.}
Though SI is not a popular choice in practice, many advanced modern block
eigensolvers employ the basic idea of SI.

To compute \(\nev\) largest magnitude eigenvalues of \(A\) using an initial
guess \(X^{(0)}\in\mathbb{C}^{n\times\nex}\) with
\RMIS{\(\nex\geq\nev\)}{\(\nex=\nev\)}, the SI algorithm produces a sequence
of subspaces \(\mathcal X^{(j)}=\Span\{X^{(j)}\}\) by performing a Rayleigh--Ritz
process on orthogonalized \(AX^{(j-1)}\).
It has been proved (see, e.g., \cite{Demmel1997}) that~\(\mathcal X^{(j)}\)
eventually converges to the desired invariant subspace, under some mild
conditions on \(A\) and~\(X^{(0)}\).
\RMIS{However, in practice, the convergence rate depends heavily on the
eigenvalue gap, and is often unsatisfactory.}{}

\RMIS{}
{However, the convergence rate of the SI algorithm depends heavily on the
gap between the \(\nev\)th and \((\nev+1)\)st largest
eigenvalues~\cite[Chapter 4]{Parlett1998}.
In practice, especially when the \(\nev\)th and \((\nev+1)\)st eigenvalues
belong to the same tightly clustered group of eigenvalues, this gap can be very
small, leading
to extremely slow convergence.}

\RMIS{Therefore}{To address this issue}, people usually use a search space
whose dimension (or block size) \(\nex\) is larger than \(\nev\) \RMIS{to
obtain a better convergence rate.}{to reduce the impact of eigenvalue
clusters on convergence.}
As proved in~\RMIS{[1,Chapter 8]}{\cite[Chapter 14]{Parlett1998}}, a
larger \(\nex\) can significantly reduce the number of iterations required
for convergence; see also Figure~\ref{fig:intro1}~(left).
\RMIS{}
{Actually, the idea of keeping a larger eigenvalue gap to avoid the influence
of clusters has also been used in other eigensolvers, such as the block
Lanczos algorithm~\cite{Ye1996} and the Davidson algorithm~\cite{SSW1998}.}

However, the cost of orthogonalization and sparse matrix--vector
multiplication (SpMV) in each iteration \RMIS{increases at the same time}
{also increases as we enlarge the search space}, which means that in
order to achieve a better performance, it is necessary to make a
trade-off between the cost per iteration and the total number of iterations.
\RMIS{}{See Table~\ref{tab:nexs} in Section~\ref{sub-sec:numexp-block} for
some experimental results.}

\begin{figure}[!tb]
\centering
\ifsimax
\includegraphics[height=6.3cm]{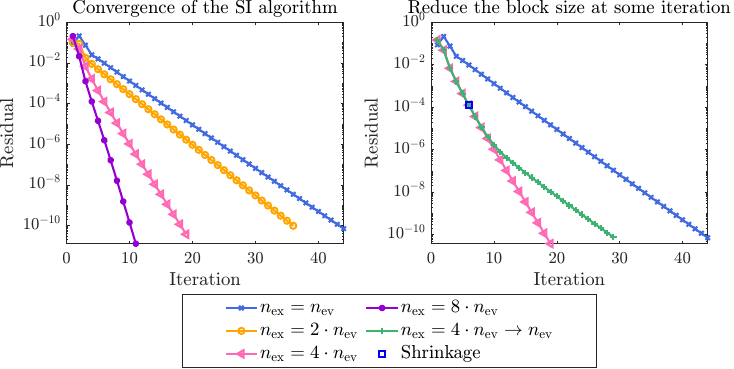}
\else
\includegraphics[height=7cm]{intro1}
\fi
\caption{Left: Use the SI algorithm to approximate five largest eigenpairs
of the matrix \texttt{bcsstk08}.
Here, \(\nev=5\) is the number of desired eigenpairs, and \(\nex\) is the
block size.
The number of iterations to converge decreases rapidly when \(\nex\) becomes
larger.
Right: Use a block size of \(\nex=4\cdot\nev\) for the first five iterations,
and then reduce the block size down to \(\nex=\nev\) by selecting eigenvectors
corresponding to the \(\nex\) largest eigenvalues.
The high convergence rate is maintained for a few more iterations.}
\label{fig:intro1}
\end{figure}

In our recent work~\cite{LSS2024}, we found that for the SI algorithm with
block size \(\nex>\nev\), if we cut the block size \(\nex\) down to \(\nev\)
at a certain point, the convergence rate does not immediately decline.
Instead, it maintains the high-speed convergence for several iterations
before gradually slowing down; see Figure~\ref{fig:intro1}~(right) for an
example.
In this example, \RMIS{even we}{even when we} reduced the block size after the \RMIS{fifth}{\(5\)th} iteration,
the residual curve keeps its original slope nearly unchanged for another five
iterations, and does not decline towards the asymptotic convergence rate with
\(\nex=\nev\) before the \(15\)th iteration.%
\footnote{The block sizes in Figures~\ref{fig:intro1} and~\ref{fig:intro2} are
chosen unnecessarily large to illustrate the idea.
More practical experiments can be found in Section~\ref{sec:numexp}.}
This phenomenon suggests that it is possible to reduce the block size to
decrease the cost of orthogonalization and SpMV per iteration, while
maintaining a relatively high convergence rate.

We remark that there are several existing techniques that reduce the
computational cost by reducing the size of the search space.
The most widely used one is the so-called \emph{deflation} technique,
including hard locking~\cite{Parlett1998},
soft locking~\cite{DSYG2018,KALO2007}, etc.
However, in the deflation process, the block size is not reduced until some
eigenpairs have converged, which has limited impact on the overall convergence
of the eigensolver.
Some other techniques, such as moving-window~\cite{SVX2015,Xue2018} and
spectral slicing~\cite{LXE2019,XLS2018}, also reduce the size of the search
space by dividing the large desired invariant subspace into several smaller
ones.

Unlike spectral slicing or moving-window, in the shrink-and-expand technique
the dimension of the search space is reduced \RMIS{}{aggressively }while the
desired invariant subspace is not divided into smaller ones.
Certainly, the convergence rate eventually declines if we keep iterating with
a smaller block size.
Thus, a natural idea is to periodically increase the block size in order to
bring the convergence rate back to a higher level.
These ideas can be illustrated in Figure~\ref{fig:intro2}, where the
shrinkage and the expansion are integrated together.

\begin{figure}[!tb]
\centering
\includegraphics[height=6cm]{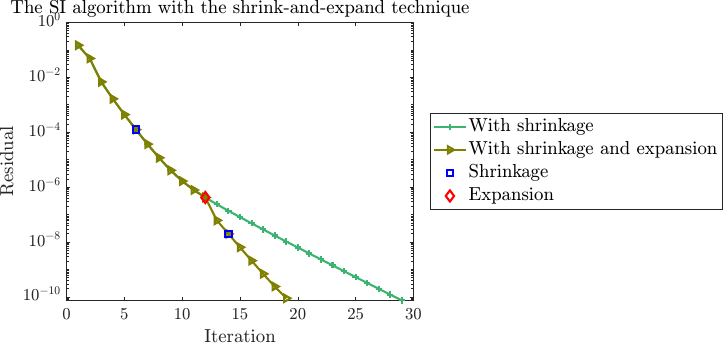}
\caption{Use the SI algorithm to approximate five largest eigenpairs of the
matrix \texttt{bcsstk08}, with \(\nex=4\cdot\nev\).
\RMIS{The}{Different from the curve with \(\nex=4\cdot\nev\) in
Figure~\ref{fig:intro1}, the}
block size is increased at the \RMIS{twelfth}{\(12\)th} iteration and, again, is decreased
at the \(14\)th iteration.
The convergence rate recovers rapidly after the expansion.}
\label{fig:intro2}
\end{figure}

The contributions of this work can be summarized as follows.
We generalize the shrink-and-expand technique to a wide range of block
eigensolvers, and further propose several strategies to perform the
shrink-and-expand technique adaptively in the computation.
\RMIS{A brief guideline of theoretical proof is also provided to stress the
effectiveness of this technique.}
{Both an illustrative example and a detailed proof are provided to stress the
effectiveness of this technique theoretically.}
Numerical experiments demonstrate that the shrink-and-expand technique can
result in a speedup of \(20\%\)--\(30\%\) on average and up to \(50\%\),
in block eigensolvers such as SI, SD, LOBPCG, and TraceMIN.

The remainder of this paper is organized as follows.
In Section~\ref{sec:prelim}, we briefly review several block eigensolvers,
including SI, SD, LOBPCG, and TraceMIN.
Then in Section~\ref{sec:algo}, we propose strategies to perform the
shrink-and-expand technique adaptively, and discuss some implementation
details.
\RMIS{Theoretical analysis is provided in Section~\ref{sec:3dexample}, and
numerical experiments are shown in Section~\ref{sec:numexp} to illustrate the
effectiveness and efficiency of the shrink-and-expand technique.}
{The theoretical analysis is divided into two sections.
In Section~\ref{sec:3dexample}, we first illustrate the fundamental logic of
the shrink-and-expand technique using a simple and intuitive example.
Then a detailed proof is provided for the general cases in
Section~\ref{sec:conv}.
Lastly, numerical experiments are shown in Section~\ref{sec:numexp} to
illustrate the effectiveness and efficiency of the shrink-and-expand
technique.}

\section{Preliminaries on block eigensolvers}
\label{sec:prelim}
In~\cite{LSS2024}, \RMIS{the shrink-and-expand technique has
been shown to be effective in the FEAST algorithm.}
{a specific shrinkage is applied at the second iteration of the FEAST
algorithm, and is shown to be effective.}
To illustrate that \RMIS{this}{the shrink-and-expand} technique can be
applied to a wider range of block eigensolvers and achieve a significant
acceleration, we selected four representative algorithms to implement
this technique in this paper.
In the following we make a brief introduction to these algorithms.

\subsection{Subspace iteration}
\label{sub-sec:si}
The subspace iteration (SI) algorithm is one of the most classical block
eigensolvers \RMIS{}{that computes a few largest eigenvalues in magnitudes
and the corresponding eigenvectors}~\cite{Demmel1997,Saad2011}.
As already mentioned, the SI algorithm iteratively obtain \(X^{(j+1)}\) by
orthogonalizing \(AX^{(j)}\), where \(X^{(j)}\in\mathbb{C}^{n\times\nex}\),
\(j=0\), \(1\), \(\dotsc\), and \(X^{(0)}\) is an initial guess.
Then, a Rayleigh--Ritz process can be applied on \(X^{(j+1)}\) to obtain the
approximate eigenpairs~\cite[Section 5.2]{Saad2011}.
In this paper, we always assume that \RMIS{}{the columns of} \(X^{(j)}\), \(j>0\), \RMIS{is}{are} already the
approximate eigenvectors obtained from the Rayleigh--Ritz process and are
arranged in ascending order corresponding to the Ritz values.
\RMIS{Detailed algorithm of the SI algorithm is presented in
Algorithm~\ref{alg:SI}.}{}

\begin{algorithm}[!tb]
\caption{The SI algorithm with the shrink-and-expand technique.}
\label{alg:SI}
\begin{algorithmic}[1]
\REQUIRE A Hermitian matrix \(A\in\mathbb{C}^{n\times n}\), an initial guess
\(X^{(0)}\in\mathbb{C}^{n\times\nex}\), a shift \(\zeta\in\mathbb{C}\), the
number of desired eigenpairs \(\nev\), the number of vectors kept after the
shrinkage \(\nes\).
\ENSURE The diagonal matrix \(\Lambda\) contains the \RMIS{computed}{}
eigenvalues \RMIS{}{closest to \(\zeta\)},
and \(X\in\mathbb{C}^{n\times\nev}\) contains the corresponding
\RMIS{computed}{} eigenvectors.
\STATE Compute \(X\) by orthogonalizing \(X^{(0)}\)\RMIS{ by QR decomposition}
{}.
\STATE \(A_{\rm{p}}\gets X\herm AX\).
\STATE Compute spectral decomposition \(A_{\rm{p}}=Z\Lambda Z\herm\),
where \(\Lambda\) has diagonals with ascending magnitudes.
\STATE \(X\gets XZ\).
\FOR{\(j=1,2,\dotsc\) until convergence}
  \STATE \(R=AX-X\Lambda\).
  \STATE Check convergence.
  \IF{\(\nev\) smallest eigenpairs have converged}
    \RETURN \(\Lambda\gets\Lambda(1:\nev,1:\nev)\), \(X\gets X(:,1:\nev)\).
  \ENDIF
  \STATE \RMIS{\(X\gets A^{-1}X\).}{\(X\gets (A-\zeta I)^{-1}X\).}
  \label{alg-step:solvesi}
  \IF{\(\mathtt{ifexpand()}\)}
    \STATE \(X\gets [X,\Xdrop]\).
  \ENDIF
  \STATE Orthogonalize \(X\)\RMIS{ by QR decomposition}{}.
  \STATE \(A_{\rm{p}}\gets X\herm AX\).
  \STATE Compute spectral decomposition \(A_{\rm{p}}=Z\Lambda Z\herm\),
  where \(\Lambda\) has diagonals with ascending magnitudes.
  \STATE \(X\gets XZ\).
  \IF{\(\mathtt{ifshrink()}\)}
    \STATE \(\Lambda\gets\Lambda(1:\nes:1:\nes)\), \(\Xdrop\gets X(:,\nes+1:{\tt{end}})\), \(X\gets X(:,1:\nes)\).
  \ENDIF
\ENDFOR
\end{algorithmic}
\end{algorithm}

The convergence of the SI algorithm depends heavily on the eigenvalue
distribution.
If we arrange the eigenvalues of \(A\) in ascending orders as
\(\lvert\lambda_1\rvert\leq\lvert\lambda_2\rvert\leq\dotsb
\leq\lvert\lambda_n\rvert\),
it can be proved that the asymptotic convergence rate of the SI algorithm
is \(\lvert\lambda_{n-\nev+1}/\lambda_{n-\nex}\rvert\)\RMIS{~[1]}
{~\cite{Parlett1998}}, where \(\nev\) is the number of desired eigenvalues,
and \(\nex\) is the block size.
Consequently, a larger block size generally corresponds to a wider eigenvalue
gap, thereby leading to faster convergence.

\RMIS{Moreover, this also indicates that the shift-and-invert technique can be
employed in the SI algorithm to compute eigenpairs closest to
\(\sigma\in\mathbb C\), like \((A-\sigma I)^{-1}X^{(j)}\).
For instance, we can apply SI to \(A^{-1}\) to compute the smallest
eigenvalues of a Hermitian and positive definite \(A\).}
{In practice, the SI algorithm is often employed jointly with the
shift-and-invert technique.
To compute a few eigenvalues closest to \(\zeta\in\mathbb C\), we replace
\(AX^{(j)}\) by \((A-\zeta I)^{-1}X^{(j)}\).
The SI algorithm with shift-and-invert is presented in Algorithm~\ref{alg:SI}.%
\footnote{To save space, we only provide the algorithms incorporated with the
shrink-and-expand technique in Algorithm~\ref{alg:SI}--\ref{alg:TM}.
For the original algorithms without this technique, readers may simply assume
that \texttt{ifshrink()} and \texttt{ifexpand()} always return
\texttt{false}.}
In Section~\ref{sec:numexp}, we always use a zero shift (i.e., \(\zeta=0\))
to compute the smallest eigenvalues of a Hermitian and positive definite
\(A\).
And the operation \(X\gets A^{-1}X\) in SI is solved by the sparse direct
method.
We remark that sparse factorization only needs to be performed once
the first solves---subsequent solves can make use of the existing
factorization.}

\subsection{Steepest descent}
The steepest descent (SD) algorithm~\cite{N2012} is also a classical
eigensolver \RMIS{}{that computes a few (algebraically) smallest eigenpairs}.
Its block version starts from an initial guess \(X^{(0)}\) with orthogonal
columns and whose (\(j+1\))st iteration takes the form
\[
X^{(j+1)} = X^{(j)}C_1^{(j)} + R^{(j)}C_2^{(j)},
\]
where \(R^{(j)} = AX^{(j)}-X^{(j)}\Lambda^{(j)}\) with
\(\Lambda^{(j)} = (X^{(j)})\herm AX^{(j)}\).
The matrices \(C_1^{(j)}\) and \(C_2^{(j)}\) are computed from the
Rayleigh--Ritz process on \(\Span\{X^{(j)},R^{(j)}\}\) so that
\(X^{(j+1)}\) consists of the Ritz vectors with respect to the \(\nex\)
smallest Ritz values.
When a preconditioner \(T\) is available, the search space
\(\Span\{X^{(j)},R^{(j)}\}\) can be replaced by
\(\Span\{X^{(j)},TR^{(j)}\}\).
We summarized the SD algorithm in Algorithm~\ref{alg:SD}.

\begin{algorithm}[!tb]
\caption{The SD algorithm with the shrink-and-expand technique.}
\label{alg:SD}
\begin{algorithmic}[1]
\REQUIRE A Hermitian matrix \(A\in\mathbb{C}^{n\times n}\), an initial guess
\(X^{(0)}\in\mathbb{C}^{n\times \nex}\), the preconditioner \(T\), the number
of desired eigenpairs \(\nev\), and the number of vectors kept after
the shrinkage \(\nes\).
\ENSURE The diagonal matrix \(\Lambda\) contains the \RMIS{computed}{smallest}
eigenvalues,
and \(X\in\mathbb{C}^{n\times\nev}\) contains the corresponding
\RMIS{computed}{} eigenvectors.

\STATE Compute \(X\) by orthogonalizing \(X^{(0)}\)\RMIS{ by QR decomposition}{}.
\STATE \(A_{\rm{p}}\gets X\herm AX\).
\STATE Compute spectral decomposition \(A_{\rm{p}}=Z\Lambda Z\herm\),
where \(\Lambda\) has diagonals with ascending magnitudes.
\STATE \(X\gets XZ\)\RMIS{}{, \(\nnow\gets\nex\)}.
\FOR{\(j=1,2,\dotsc\) until convergence}
  \STATE \(R\gets AX-X\Lambda\).
  \IF{\(\nev\) smallest eigenpairs have converged}
    \RETURN \(\Lambda\gets\Lambda(1:\nev,1:\nev)\), \(X\gets X(:,1:\nev)\).
  \ENDIF
  \IF{\(\mathtt{ifexpand()}\)}
    \STATE \(X\gets [X, \Xdrop]\).
    \RMIS{}{\STATE \(\nnow\gets\nex\).}
  \ENDIF
  \STATE \(W\gets TR\).
  \STATE \(W\gets W-X(X\herm W)\) and orthogonalize \(W\)\RMIS{ by QR decomposition}{}.
  \STATE \(S\gets[X,W]\).
  \STATE \(A_{\rm{p}}\gets S\herm AS\).
  \STATE Compute spectral decomposition \(A_{\rm{p}}=Z\Lambda Z\herm\),
  where \(\Lambda\) has diagonals with ascending magnitudes.
  \STATE \(X\gets SZ(:,1:\RMIS{\nex}{\nnow})\).
  \IF{\(\mathtt{ifshrink()}\)}
    \STATE \(\Xdrop\gets X(:, \nes+1:{\tt{end}})\).
    \STATE \(\Lambda\gets\Lambda(1:\nes, 1:\nes)\), \(X\gets X(:, \RMIS{\nes}{1:\nes})\).
    \RMIS{}{\STATE \(\nnow\gets\nes\).}
  \ENDIF
\ENDFOR
\end{algorithmic}
\end{algorithm}

\subsection{LOBPCG}
The locally block optimal preconditioned conjugate gradient (LOBPCG)
algorithm~\cite{KALO2007,Knyazev2001} is a popular block eigensolver that
computes \RMIS{the}{a few (algebraically)} smallest eigenpairs of a large
Hermitian matrix~\(A\),
especially when a good preconditioner \(T\) is available.
Starting from an initial guess \(X^{(0)}\) with orthogonal columns, the
(\(j+1\))st iteration of the LOBPCG algorithm takes the form
\[
X^{(j+1)} = X^{(j)}C_1^{(j)} + X^{(j-1)}C_2^{(j)} + W^{(j)}C_3^{(j)},
\]
where \(W^{(j)} = T(AX^{(j)}-X^{(j)}\Lambda^{(j)})\) with
\(\Lambda^{(j)} = X^{(j)*} AX^{(j)}\).
Similarly to the SD algorithm, the matrices \(C_1^{(j)}\), \(C_2^{(j)}\), and
\(C_3^{(j)}\) are chosen optimally from a Rayleigh--Ritz process on the
search space \(\Span\{X^{(j)},X^{(j-1)},W^{(j)}\}\) to obtain the
\(\nex\) smallest Ritz pairs.

\RMIS{We remark that a practical implementation of LOBPCG is usually quite
complicated, aiming to make a trade-off between performance and numerical
stability.
See~\cite{DSYG2018} for implementation details, including an improved
Hetmaniuk--Lehoucq (HL) trick, and soft-locking-based deflation.
We summarize the LOBPCG algorithm in Algorithm~\ref{alg:LOBPCG}.}
{To better reflect the performance of the algorithm in practice, we employ an
improved Hetmaniuk--Lehoucq trick and soft locking in our implementation of
the LOBPCG algorithm.
The improved Hetmaniuk--Lehoucq trick here is a basis selection strategy for
LOBPCG that produces an orthogonal basis \([X^{(j+1)},P^{(j+1)}]\) of
\(\Span\{X^{(j+1)},X^{(j)}\}\) by a clever manipulation of the output of the
Rayleigh--Ritz process;
see~\cite[Section~4.2]{DSYG2018} for details.
We summarize the algorithm in Algorithm~\ref{alg:LOBPCG}.
}

\begin{algorithm}[!tb]
\caption{The LOBPCG algorithm with the shrink-and-expand technique.}
\label{alg:LOBPCG}
\begin{algorithmic}[1]
\REQUIRE A Hermitian matrix \(A\in\mathbb{C}^{n\times n}\), an initial guess
\(X^{(0)}\in\mathbb{C}^{n\times\nex}\), the preconditioner \(T\), the number
of desired eigenpairs \(\nev\), and the number of vectors kept after the
shrinkage \(\nes\).
\ENSURE The diagonal matrix \(\Lambda\) contains the \RMIS{computed}{smallest}
eigenvalues,
and \(X\in\mathbb{C}^{n\times\nev}\) contains the corresponding
\RMIS{computed}{} eigenvectors.
\STATE Compute \(X\) by orthogonalizing \(X^{(0)}\)\RMIS{ by QR decomposition}{}.
\STATE \(A_{\rm{p}}\gets X\herm AX\).
\STATE Compute spectral decomposition \(A_{\rm{p}}=Z\Lambda Z\herm\),
where \(\Lambda\) has diagonals with ascending magnitudes.
\STATE \(X\gets XZ\).
\STATE \(P\gets [~]\)\RMIS{}{, \(\nnow\gets\nex\)}.
\FOR{\(j=1,2,\dotsc\) until convergence}
  \STATE \(R\gets AX-X\Lambda\).
  \IF{\(\nev\) smallest eigenpairs have converged}
    \RETURN \(\Lambda\gets\Lambda(1:\nev,1:\nev)\), \(X\gets X(:,1:\nev)\).
  \ELSE
    \STATE Deflate soft-locked columns from \(P\) and \(R\)\RMIS{}{~and update \(\nnow\) accordingly}.
  \ENDIF
  \STATE \(W\gets TR\).
  \STATE \(W\gets W-[X, P]\bigl([X, P]\herm W\bigr)\) and orthogonalize \(W\)\RMIS{ by QR decomposition}{}.
  \STATE \(S\gets[X,P,W]\).
  \IF{\(\mathtt{ifexpand()}\)}
    \STATE \(\Sdrop\gets\Sdrop-S(S\herm \Sdrop)\) and orthogonalize \(\Sdrop\)\RMIS{ by QR decomposition}{}.
    \STATE \(X\gets [X,\Sdrop(:,1:{\tt{end}}/2)]\), \(P\gets [\RMIS{X}{P},\Sdrop(:,{\tt{end}}/2+1:{\tt{end}})]\), \(S\gets [X,P,\RMIS{R}{W}]\).
    \RMIS{}{\STATE \(\nnow\gets\nex\).}
  \ENDIF
  \STATE \(A_{\rm{p}}\gets S\herm AS\).
  \STATE Compute spectral decomposition \(A_{\rm{p}}=Z\Lambda Z\herm\),
  where \(\Lambda\) has diagonals with ascending magnitudes.
  \STATE \RMIS{\([C_X,C_{\rm{p}}]={\tt{HLtrick}}(Z)\)}
    {\([X,P]={\tt{HLtrick}}(S,Z)\) \quad \textsl{\% The improved
    Hetmaniuk--Lehoucq trick}}
  \IF{\(\mathtt{ifshrink()}\)}
    \STATE \(\Sdrop\gets [X(:,\nes+1:{\tt{end}}),P(:,\nes+1:{\tt{end}})]\).
    \STATE \(\Lambda\gets\Lambda(1:\nes\RMIS{:}{,}1:\nes)\), \(X\gets X(:,1:\nes)\), \(P\gets P(:,1:\nes)\).
    \RMIS{}{\STATE \(\nnow\gets\nes\).}
  \ENDIF
\ENDFOR
\end{algorithmic}
\end{algorithm}

\subsection{TraceMIN}
The trace minimization (TraceMIN) algorithm~\cite{KSS2013,SW1982} is 
\RMIS{based on
Fan's trace minimization principle
\[
\sum_{i=1}^\nev\lambda_i
=\min_{\substack{X\in\mathbb{C}^{n\times\nev} \\ X\herm X=I}}
\trace(X\herm AX),
\]
where \(\lambda_1\leq\lambda_2\leq\dotsb\leq\lambda_n\) are the eigenvalues of \(A\).}
{designed to compute a few eigenpairs with smallest magnitudes for Hermitian matrices.}
\RMIS{
If one can sequentially take \(X^{(j+1)}\) such that
\(\trace(X^{(j+1)*} AX^{(j+1)})<\trace(X^{(j)*} AX^{(j)})\),
\(\Span\{X^{(j+1)}\}\) serves as a better approximation to the desired invariant
subspace compared to \(\Span\{X^{(j)}\}\).
To this end, 
}{
In each iteration, }the TraceMIN algorithm \RMIS{constructs}{updates \(X^{(j+1)}\)} with
\(X^{(j+1)}=X^{(j)}-\Delta^{(j)}\), where \(\Delta^{(j)}\) is the solution of
\begin{equation}
\label{eq:tracminarg}
\argmin_{(X^{(j)})\herm\Delta^{(j)}=0}
\trace(X^{(j)}-\Delta^{(j)})\herm
A(X^{(j)}-\Delta^{(j)}).
\end{equation}
In the classical implementation of the TraceMIN algorithm~\cite{SW1982},
the KKT conditions are employed to solve the minimization
problem~\eqref{eq:tracminarg}.
Also in~\cite{SW1982}, it has been shown that~\eqref{eq:tracminarg} can be
converted to
\[
P_X^{(j)}AP_X^{(j)}\Delta^{(j)}=P_X^{(j)}AX^{(j)},
\qquad P_X=I-X^{(j)}(X^{(j)})\pinv.
\]

One of the highlights of the TraceMIN algorithm is that
\(P_X^{(j)}AP_X^{(j)}\) does not need to be solved accurately\RMIS{ ---}{.}
\RMIS{}{In other words, }the computational cost can be reduced by using an
inexact linear solver\RMIS{.}{,
which is similar to that of the Jacobi--Davidson algorithm~\cite{HJ2023,SV2000}.}
In Section~\ref{sec:numexp}, we employ the conjugate gradient (CG)
algorithm with only five iterations on solving the linear systems\RMIS{}{~\cite{KSS2013}}.
The classical TraceMIN algorithm is listed in Algorithm~\ref{alg:TM}.

\begin{algorithm}[!tb]
\caption{The TraceMIN algorithm with the shrink-and-expand technique.}
\label{alg:TM}
\begin{algorithmic}[1]
\REQUIRE A Hermitian matrix \(A\in\mathbb{C}^{n\times n}\),
an initial guess \(X^{(0)}\in\mathbb{C}^{n\times\nex}\), the number of
desired eigenpairs \(\nev\), and the number of vectors kept after the
shrinkage \(\nes\).
\ENSURE The diagonal matrix \(\Lambda\) contains the \RMIS{computed}{}
eigenvalues \RMIS{}{with smallest magnitudes},
and \(X\in\mathbb{C}^{n\times\nev}\) contains the corresponding
\RMIS{computed}{} eigenvectors.
\STATE Compute \(X\) by orthogonalizing \(X^{(0)}\)\RMIS{ by QR decomposition}{}.
\STATE \(A_{\rm{p}}\gets X\herm AX\).
\STATE Compute spectral decomposition \(A_{\rm{p}}=Z\Lambda Z\herm\),
where \(\Lambda\) has diagonals with ascending magnitudes.
\STATE \(X\gets XZ\).
\FOR{\(j=1,2,\dotsc\) until convergence}
  \STATE \(R\gets AX-X\Lambda\).
  \IF{\(\nev\) smallest eigenpairs have converged}
    \RETURN \(\Lambda\gets\Lambda(1:\nev,1:\nev)\), \(X\gets X(:,1:\nev)\).
  \ENDIF
  \STATE Solve \(P_XAP_X\Delta=P_XR\) with inexact linear solver, where \(P_X=I-X\herm X\).
  \STATE \(X\gets X-\Delta\).
  \IF{\(\mathtt{ifexpand()}\)}
    \STATE \(X\gets[X,\Xdrop]\).
  \ENDIF
  \STATE Orthogonalize \(X\)\RMIS{ by QR decomposition}{}.
  \STATE \(A_{\rm{p}}\gets X\herm AX\).
  \STATE Compute spectral decomposition \(A_{\rm{p}}=Z\Lambda Z\herm\),
  where \(\Lambda\) has diagonals with ascending magnitudes.
  \STATE \(X\gets XZ\).
  \IF{\(\mathtt{ifshrink()}\)}
    \STATE \(\Xdrop\gets X(:,\nes+1:{\tt{end}})\).
    \STATE \(\Lambda\gets\Lambda(1:\nes,1:\nes)\), \(X\gets X(:,1:\nes)\).
  \ENDIF
\ENDFOR
\end{algorithmic}
\end{algorithm}

\section{The shrink-and-expand technique}
\label{sec:algo}
\subsection{A general framework}
In Section~\ref{sec:intro}, we have briefly introduced the key idea of the
shrink-and-expand technique.
In the following, we provide a more detailed description.

For a block eigensolver, suppose \(X\in\mathbb{C}^{n\times\nex}\) consists of
the approximate eigenvectors.
\RMIS{The \emph{shrinkage} process is to drop a few columns from \(X\) so
that the reduced matrix \(X\in\mathbb{C}^{n\times\nes}\) (with
\(\nev\leq\nes<\nex\)) is used in subsequent iterations.
Conversely, the \emph{expansion} process is to append a few columns to \(X\)
so that the new \(X\) is \(n\times\nes\).}
{Then shrinkage and expansion are defined as follows.
\begin{itemize}
\item \emph{Shrinkage:} Dropping \(\nex-\nes\) (\(\nev\leq\nes<\nex\))
columns from \(X\in\mathbb{C}^{n\times\nex}\) to obtain a new
\(X\in\mathbb{C}^{n\times\nes}\).
And using this smaller \(X\) in subsequent iterations to save computational
cost.
\item \emph{Expansion:} Appending \(\nex-\nes\) linear independent vectors to
\(X\in\mathbb{C}^{n\times\nes}\) to obtain a new
\(X\in\mathbb{C}^{n\times\nex}\).
And using this larger \(X\) in subsequent iterations to increase convergence
rate.
\end{itemize}}
Typically, the vectors appended in expansion are those who dropped from the
last shrinkage.\RMIS{}{%
\footnote{Appending these vectors is of high efficiency. More choices will be
introduced in Section~\ref{sub-sec:exp-newvec}.}}
A flowchart to illustrate the shrink-and-expand technique in a block
eigensolver is shown in Figure~\ref{fig:flowchart}.

\begin{figure}[!tb]
\centering
\includegraphics[width=0.8\textwidth]{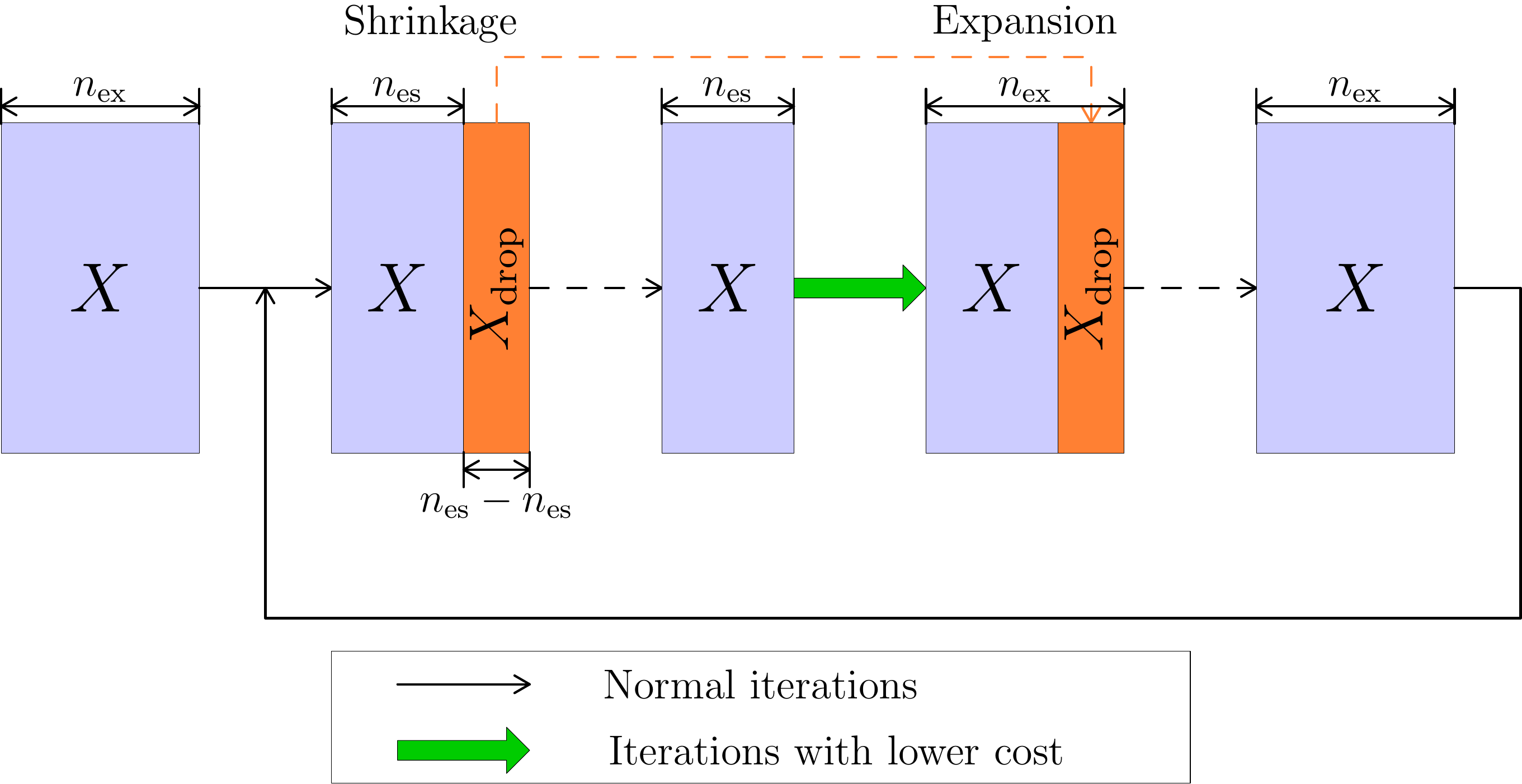}
\caption{The process of employing the shrink-and-expand technique in block
eigensolvers.
The size of \(X\) decreases from \(n\times\nex\) to \(n\times\nes\) after
shrinkage, allowing the iterations marked with the thick arrow to benefit from
the reduced cost.
To prevent the convergence rate from deteriorating, a larger block size is
restored periodically.}
\label{fig:flowchart}
\end{figure}

To employ the shrink-and-expand technique in block eigensolvers, we provide a
general framework in Algorithm~\ref{alg:framework}.
The framework is straightforwardly applied to SI, SD, and TraceMIN.
For the LOBPCG algorithm, there is a slight difference that, since both \(X\)
and \(P\) are stored in the memory, we drop/append columns for \(X\) and \(P\)
simultaneously; see Algorithm~\ref{alg:LOBPCG}.\RMIS{}{%
\footnote{The implementation of the shrink-and-expand technique is not unique.
Users may also choose to expand or shrink only \(X\).
We choose to adjust \(X\) and~\(P\) simultaneously for a simpler projection step.}}

\begin{algorithm}[!tb]
\caption{A block eigensolver employing the shrink-and-expand technique.}
\label{alg:framework}
\begin{algorithmic}[1]
\REQUIRE The matrix \(A\), the initial guess \(X^{(0)}\).
\ENSURE The approximate eigenpairs \((\Lambda,X)\).
\STATE Obtain an approximation \((\Lambda,X)\) by the Rayleigh--Ritz
process on \(X^{(0)}\).
\FOR{\(j=1\), \(2\), \(\dotsc\) until convergence}
  \STATE Check convergence.
  \STATE Update \(X\) (e.g., \(X\gets A^{-1}X\) for the SI algorithm).
  \IF{\(\mathtt{ifexpand()}\)}
    \STATE \(X\gets [X,\Xdrop]\).
  \ENDIF
  \STATE Construct the search space \(\mathcal S\) by \(X\)
(possibly, also by other information).
  \STATE Obtain the approximate eigenpairs \((\Lambda,X)\)
by the Rayleigh--Ritz process on \(\mathcal S\).
  \IF{\(\mathtt{ifshrink()}\)}
    \STATE \([X,\Xdrop]\gets X\).
  \ENDIF
\ENDFOR
\end{algorithmic}
\end{algorithm}

\RMIS{}{
We remark that ideally most calls to \texttt{ifshrink()}/\texttt{ifexpand()}
return \texttt{false}, meaning that shrinkage or expansion is employed only in
a few iterations, with most of iterations involving neither.
After a shrinkage step, the algorithm proceeds with as many as possible
cheaper iterations before an expansion is finally applied, thereby reducing
the overall computational expense.}

\subsection{Strategies for applying shrink-and-expand}
\label{subsec:strategies}
In Algorithms~\ref{alg:SI}--\ref{alg:TM}, we use \texttt{ifshrink()} and
\texttt{ifexpand()} to determine whether some columns need to be
dropped/appended.
In this subsection, we propose three concrete strategies.

\paragraph{The {\tt{fix}} strategy}
A straightforward strategy is to expand the block size every \(\je\)th
iterations to avoid the overall convergence rate decay, and then shrink the
block size in the \(\js\)th iteration after every expansion, where \(\je\)
and \(\js\) are prescribed by the user.
Specifically, we apply the shrinkage at the (\(i\cdot\je+\js\))th iteration and
the expansion at the (\(i\cdot\je\))th iteration, where \(i=0\), \(1\),
\(\dotsc\).
This strategy is referred to as {\tt{fix}} and is summarized in
Algorithm~\ref{alg:fix}.\RMIS{}{%
\footnote{We use the MATLAB code \(\texttt{mod}(a,b)\), which return \(0\)
if and only if \(a\) is divisible by \(b\).}}
Such a strategy can ensure a stable frequency for the shrinkage and the expansion,
and works well when some a priori estimates are available to determine
reasonable values for \(\js\) and \(\je\).

\begin{algorithm}[!tb]
\caption{The shrinkage and expansion strategy---{\tt{fix}}.}
\label{alg:fix}
\begin{algorithmic}[1]
\REQUIRE T\RMIS{}{he dimension of the current search space \(\nnow\),
the number of vectors kept after the shrinkage \(\nes\), t}he iteration now
\(j\), the period of taking expansion \(\je\), and the number of iterations
of taking shrinkage after the last expansion \(\js\).
\ENSURE \RMIS{}{A logical variable indicating \texttt{true} or \texttt{false}.}
\STATE \textbf{function} \texttt{ifshrink()}
\STATE \quad\textbf{return} \RMIS{whether \(j-\js\) is a multiple of \(\je\)}{whether \(\texttt{mod}(j-\js,\je)=0\)}
\STATE \textbf{end function}
\STATE \textbf{function} \texttt{ifexpand()}
\STATE \quad\textbf{return} \RMIS{whether \(j\) is a multiple of \(\je\)}
{whether \(\nnow=\nes\) \AND \(\texttt{mod}(j,\je)=0\)}
\STATE \textbf{end function}
\end{algorithmic}
\end{algorithm}

\paragraph{The {\tt{slope}} strategy}
The {\tt{fix}} strategy might not be the best choice when there is no
\RMIS{useful information}{information can be used} for choosing \(\js\) and
\(\je\).
In such case, if \(\js/\je\) is too large, it may result in an unsatisfactory
overall convergence rate.
Conversely, if \(\js/\je\) too small, we might miss out on potential
acceleration from the shrink-and-expand technique.
Therefore, we propose the {\tt{slope}} strategy aiming to use a small block
size for as many iterations as possible while avoiding significantly
reducing the convergence rate.
It is natural to adaptively determine the timing by the convergence rate
itself.

We define the slope of each iteration as
\begin{equation}
\label{eq:ckslope}
c^{(j)}=\log_{10}r^{(j-1)}-\log_{10}r^{(j)},\qquad j=1, 2, \dotsc,
\end{equation}
where \(r^{(j)}\) is the relative residual of the \(j\)th iteration.%
\footnote{The definition of \(r^{(j)}\) depends on specific algorithms
and also the user's settings.
We provide our definition of \(r^{(j)}\) in the experiments in
Section~\ref{sec:numexp}.}
If \(c^{(j)}\) is large, which means that the eigensolver is converging
rapidly, we keep the block size small;
if \(c^{(j)}\) becomes smaller, indicating that the convergence rate is
deteriorating, then an expansion is needed.
In practice, we record the maximum slope as \(\cmax\)\RMIS{ every time}{}
after the last shrinkage, and compare it with the slope of the current
iteration, \(c^{(j)}\).
If \RMIS{\(c^{(j)}<\mu\cmax\)}{\(\cmax/c^{(j)}>\mu\)},
where \RMIS{\(\mu<1\)}{\(\mu>1\)} is a tunable threshold, we take an
expansion right after this iteration because the convergence rate has
already deteriorated.
In principle, we can also use the slope to determine where to employ the
shrinkage.
However, one or two expanded iterations can typically restore the high
convergence rate.
Thus, using a fixed~\(\js\) for taking the next shrinkage suffices, just as
the {\tt{fix}} strategy does.
We summarize this strategy, named {\tt{slope}}, in Algorithm~\ref{alg:slope}.

\begin{algorithm}[!tb]
\caption{The shrinkage and expansion strategy---{\tt{slope}}/{\tt{slopek}}.}
\label{alg:slope}
\begin{algorithmic}[1]
\REQUIRE T\RMIS{}{he dimension of the current search space \(\nnow\), the number
of vectors kept after the shrinkage \(\nes\), t}he number of iterations since
the last expansion \(\jl\), the number of iterations of taking shrinkage
after the last expansion \(\js\), the slope of the current iteration \(c\),
the largest slope since the last shrinkage \(\cmax\), and the threshold of
the expansion \(\mu\).
\ENSURE \RMIS{}{A logical variable indicating \texttt{true} or \texttt{false}.}
\STATE \textbf{function} \texttt{ifshrink()}
\STATE \quad\textbf{return} \RMIS{whether \(\jl=\je\)}{whether \(\jl=\je\)}
\STATE \textbf{end function}
\STATE \textbf{function} \texttt{ifexpand()}
\STATE \quad\textbf{return} \RMIS{whether \(c>\mu\cmax\)}
{whether \(\nnow=\nes\) \AND \(\cmax/c>\mu\)}
\STATE \textbf{end function}
\end{algorithmic}
\end{algorithm}

\paragraph{The {\tt{slopek}} strategy}
In practice, we observe that the convergence history can be even more
complicated.
For example, in the SD algorithm the residual oscillates frequently, although
the overall trend is downward (see Figure~\ref{fig:SD3str} in
Section~\ref{sec:numexp}).
In this case, the {\tt{slope}} strategy can perform poorly since many
unnecessary expansion is applied whenever there is an increase of the
residual.
Therefore, we propose the {\tt{slopek}} strategy, where the average slope over
several iterations is considered instead of just one iteration.
We redefine \(c^{(j)}\) in Algorithm~\ref{alg:slope} as
\begin{equation}
\label{eq:ckslopek}
c^{(j)}=\frac{\log_{10}r^{(j-\jp)}-\log_{10}r^{(j)}}{\jp},\qquad j\ge\jp.
\end{equation}
This strategy takes the average slope of last \(\jp\) iterations, and is less
vulnerable to the influence of fluctuations.
We provide detailed numerical experiments in Section~\ref{sub-sec:numexp-3str}
to illustrate their performance under different applications.

\RMIS{In practice, it is recommended to run a few iterations of the eigensolver
before applying the shrink-and-expand technique.}
{Readers may notice that the definitions of the three strategies all assume
that the first shrinkage has already been performed.
However, we have not specified when the first shrinkage should be taken.
In practice, we recommend running a few iterations of the eigensolver
before applying the first shrinkage.}
This type of conservative strategy, which has been used in,
e.g.,~\cite{SAY2018}, allows the eigensolvers to achieve a reasonably stable
convergence rate, thereby ensuring the robustness of the proposed technique.
\RMIS{}{In our implementation, both a tolerance on the minimum iteration
\(\jwarm\) and a tolerance on the maximum residual \(\rwarm\) are employed.
The first shrinkage will be employed as soon as both \(j\ge\jwarm\) and
\(r^{(j)}\le\rwarm\) are satisfied.}

\section{\RMIS{Convergence theory on the subspace iteration}
{An illustrative example for the convergence of the shrink-and-expand technique}}
\label{sec:3dexample}
\RMIS{}
{A rigorous analysis on the acceleration effect of the shrink-and-expand
technique is quite complicated.
Therefore, in this section, we first present a simple \(3\times3\) example to
provide a quick understanding of the fundamental principles of the
shrink-and-expand technique.
A detailed analysis for the more general cases is deferred to
Section~\ref{sec:conv}.}

\RMIS{In this section, we use the SI algorithm with shift-and-invert as an
example to show properties of the shrink-and-expand technique.
There are two things to be explained in this section:}{}
\RMISBLOCK{
\begin{enumerate}
\item Why the convergence rate do not decay immediately after the shrinkage.
\item Why the convergence rate can be restored after the expansion.
\end{enumerate}
}{}

\RMIS{We first focus on the case that only the smallest eigenvalue of the
Hermitian positive definite matrix \(A\) is desired, and
\(\nex=\nev=1\).
The shift-and-invert technique is applied with a shift \(\sigma=0\), in other
words, \(A^{-1}x^{(j)}\) is employed in each iteration.}
{We consider employing the shrink-and-expand technique in the SI
algorithm (Algorithm~\ref{alg:SI}) with shift \(\zeta=0\).
Only the smallest (one) eigenpair of the Hermitian positive definite matrix
\(A\) is desired.}
To quantify convergence speed, we provided the following proposition, which
can be obtained from the proof of~\cite[Theorem 4.2.1]{Parlett1998}.
\begin{proposition}
\label{prop:palat}
Let \(A\in\mathbb{C}^{n\times n}\) be a Hermitian positive
definite matrix with normalized eigenpairs \((\lambda_1,v_1)\),
\((\lambda_2,v_2)\), \(\dotsc\), \((\lambda_n,v_n)\).
Assume that \(0<\lambda_1\leq\lambda_2\leq\dotsb\leq\lambda_n\), and the
initial guess \(x^{(0)}\in\mathbb{C}^n\) satisfying \(v_1\herm x^{(0)}\neq0\).
Then the convergence rate of the inverse power method can be represented in
\begin{equation}
\label{eq:rhopm}
\rho=\frac{\tan\theta^{(j+1)}}{\tan\theta^{(j)}}
=\lambda_1\bigl\lVert A^{-1}\check{x}^{(j)}\bigr\rVert_2,\qquad
\check{x}^{(j)}=\frac{x^{(j)}-v_1v_1\herm x ^{(j)}}{\lVert x^{(j)}-v_1v_1\herm x^{(j)}\rVert_2},
\end{equation}
where \(x^{(j)}\) is the approximate eigenvector of iteration \(j\), and
\(\theta^{(j)}\) represents the angle between \(x^{(j)}\) and
\(v_1\).
\end{proposition}

By substituting \(\bigl\lVert A^{-1}\check{x}^{(j)}\bigr\rVert_2\le1/\lambda_2\)
into~\eqref{eq:rhopm}, we obtain \(\rho\le\lambda_1/\lambda_2\),
which is the most well-known bound for the convergence rate of the power
method.
However, such a substitution is actually the worst-case estimate, where the
influence of the direction of \(x^{(j)}\) is ignored.
In practice, \(\bigl\lVert A^{-1}\check{x}^{(j)}\bigr\rVert_2\) can be much smaller than
\(1/\lambda_2\), and the power method converges much more rapidly.
For example, if \(\termgk_1^2+\termgk_2^2=1\), the initial guess
\(x^{(0)}=v_1\termgk_1+v_n\termgk_2\) converges faster than
\(y^{(0)}=v_1\termgk_1+v_2\termgk_2\) since
\(\bigl\lVert A^{-1}\check{x}^{(0)}\bigr\rVert_2=1/\lambda_n\) is much smaller than
\(\bigl\lVert A^{-1}\check{y}^{(0)}\bigr\rVert_2=1/\lambda_2\).
Typically, we have
\[
\begin{aligned}
\tan\angle(x^{(j)},v_1)
&=\tan\angle(A^{-j}x^{(0)},v_1)
=\frac{\lambda_1^j\termgk_2}{\lambda_n^j\termgk_1}\\
&\qquad\le\frac{\lambda_1^j\termgk_2}{\lambda_2^j\termgk_1}
=\tan\angle(A^{-j}y^{(0)},v_1)
=\tan\angle(y^{(j)},v_1).
\end{aligned}
\]

Unfortunately, for \RMIS{}{a} general
\begin{equation*}
x^{(0)}=v_1\termgk_1+v_2\termgk_2+\cdots+v_n\termgk_n,
\end{equation*}
\(\check{x}^{(j)}\) gradually converges to \(v_2\) as
\[
\check{x}^{(j)}=v_2\frac{\termgk_2}{\lambda_2^j\gamma_j}+\cdots+v_n\frac{\termgk_n}{\lambda_n^j\gamma_j},\qquad
\gamma_j=\sqrt{\sum_{i=2}^n\frac{\termgk_i^2}{\lambda_{i}^{2k}}}.
\]
\RMIS{}
{
Using Proposition~\ref{prop:palat}, we have
\[
\begin{aligned}
\lim_{j\to\infty}\rho
&=\lim_{j\to\infty}\lambda_1\bigl\lVert A^{-1}\check{x}^{(j)}\bigr\rVert_2\\
&=\lim_{j\to\infty}\lambda_1\left\lVert\frac{1}{\lambda_2}\cdot v_2\frac{\termgk_2}{\lambda_2^j\gamma_j}
+\cdots+\frac{1}{\lambda_n}\cdot v_n\frac{\termgk_n}{\lambda_n^j\gamma_j}\right\rVert_2.
\end{aligned}
\]
Since \(\lim_{j\to\infty}\termgk_i/(\lambda_i^j\gamma_j)=0\) for \(i=3\),
\(\dotsc\), \(n\), and
\(\lim_{j\to\infty}\termgk_2/(\lambda_2^j\gamma_j)=1\), we obtain
\[
\lim_{j\to\infty}\rho
=\lambda_1\left\lVert\frac{1}{\lambda_2}\cdot v_2\right\rVert_2
=\frac{\lambda_1}{\lambda_2},
\]
which means that the convergence rate \(\rho\) will become larger and
eventually be close to \(\lambda_1/\lambda_2\).
}

\subsection{Why the convergence rate do not decline immediately after the
shrinkage}
Now we consider the case \(\nev=1\)\RMIS{ and \(\nex=2\)}{, \(\nex=2\), and \(\nes=1\)} aiming to
explain why the convergence rate remain higher even after a shrinkage
process.

\RMIS{Let \(X^{(j)}=\bigl[x_1^{(j)},x_2^{(j)}\bigr]\) arranged in ascending order
with respect to their Ritz values.
After a shrinkage process, the vector corresponding to \(\lambda_1\), named
\(x_1^{(j)}\), is selected for subsequent iterations.
However, when employing the Rayleigh--Ritz process, the other vector in
\(X^{(j)}\), named \(x_2^{(j)}\), approximates the second smallest
eigenvector, \(v_2\).
Since these two vectors must be orthogonal to each other, the component of
\(x_1^{(j)}\) on \(v_2\) will be significantly reduced by \(x_2^{(j)}\).
As we know from Proposition~\ref{prop:palat}, even if we select \(x_1^{(j)}\)
for the subsequent iterations, the convergence rate will still remain better
than \(\lambda_1/\lambda_2\) for a while.}{}

Let us use a concrete example to illustrate the phenomenon:
\begin{equation}
\label{eq:shrinkexp}
A=\bmat{1 & & \\ & 10 & \\ & & 100}, \qquad
X^{(0)}=
\bmat{1 & 1 \\
1 & 4 \\
1 & 2}.
\end{equation}
After one iteration of the SI algorithm, we obtain the approximate
eigenvectors
\[
X^{(1)}=\bigl[x_1^{(1)},x_2^{(1)}\bigr]
\approx
\bmat{
9.9998\times10^{-1}&2.1951\times10^{-3}\\
-2.4159\times10^{-3}&9.9944\times10^{-1}\\
6.5860\times10^{-3}&3.3329\times10^{-2}
},
\]
where \(x_1^{(1)}\) and \(x_2^{(1)}\) are already sorted in ascending order
with respect to their Ritz values.
Then, the approximation of \(v_1=[1,0,0]\herm\), which is \(x_1^{(1)}\), has a
significantly smaller component on \(v_2=[0,1,0]\herm\), compared to that on
\(v_3=[0,0,1]\herm\).
By Proposition~\ref{prop:palat}, the current convergence rate of \(x_1^{(1)}\)
is \(\rho=3.5696\times10^{-2}\), which is less than half of the asymptotic
convergence rate \(\lambda_1/\lambda_2=1\times10^{-1}\).
Therefore, even if we only use \(x_1^{(1)}\) for the subsequent iteration, the
convergence rate will still be better.

\subsection{Why the convergence rate can be restored after the expansion}
Similarly, if we begin with \RMIS{\(\nex=1\)}{a one dimensional search space
(to simulate the stage after a shrinkage)} and add another vector at
some iteration with the expansion technique, the Rayleigh--Ritz method helps
to eliminate the component on \(v_2\) and to restore the convergence rate.

\RMIS{Here, we also provide an example.}{}
With the same \(A\) as in~\eqref{eq:shrinkexp}, one step of power iteration
with \(X^{(0)}=[1,1,1]\herm\) (i.e., the first column of \(X^{(0)}\)
in~\eqref{eq:shrinkexp}) yields
\[
X^{(1)}=\frac{1}{\sqrt{10101}}\bmat{100 \\ 10 \\ 1}.
\]
It can be seen that the component on \(v_2\) is ten times larger than that of
\(v_1\), so that the convergence rate \(\rho=9.95\times10^{-2}\approx1\times10^{-1}=\lambda_1/\lambda_2\).
However, if we expand \(X^{(1)}\) with the second column \([1,4,2]\herm\)
of~\eqref{eq:shrinkexp} to make
\[
X^{(1)}\gets\bmat{
100/\sqrt{10101}&1\\
10/\sqrt{10101}&4\\
1/\sqrt{10101}&2
},
\]
then after another iteration we obtain
\[
X^{(2)}=\bigl[x_1^{(2)},x_2^{(2)}\bigr]
\approx
\bmat{
1.0000\times10^{0}&-2.0324\times10^{-4}\\
2.2386\times10^{-4}&9.9870\times10^{-1}\\
-3.9883\times10^{-4}&5.0959\times10^{-2}
},
\]
where \(x_1^{(2)}\) and \(x_2^{(2)}\) are, similarly, sorted by Ritz values
in ascending order.
It can be observed that, for \(x_1^{(2)}\), the component on \(v_2\),
\(2.2386\times10^{-4}\), is of the same magnitude as that on \(v_3\), which is
\(-3.9883\times10^{-4}\).
This brings the convergence rate of \(x_1^{(2)}\) down to
\RMIS{\(4.9716\times10^{-2}\)}{\(\rho=4.9716\times10^{-2}\),
which is, again, less than half of the asymptotic convergence rate
\(\lambda_1/\lambda_2=1\times10^{-1}\)}.
Thus, even if \(x_2^{(2)}\) is dropped, the subsequent iteration on
\(x_1^{(2)}\) still converges faster.
\RMIS{}{In other words, \(X^{(2)}\) now is ready for another shrinkage.}

The discussion above reveals the principle behind the shrink-and-expand
technique.
At the beginning, we perform iterations with a larger block size to
eliminate the component of \(x_1^{(j)}\) on \(v_2\).
Then, when the shrinkage is applied, the approximation of \(v_1\), i.e.,
\(x_1^{(j)}\), is selected for subsequent iterations.
Since the component of \(x_1^{(j)}\) on \(v_2\) is relatively small,
the convergence rate keeps better than \(\lambda_1/\lambda_2\) for a few
iterations.

Later, due to the nature of the SI algorithm,
the components will gradually concentrate to~\(v_2\) again,
where a decay in the convergence rate will be observed.
Then, another expansion can be employed to eliminate this component on \(v_2\)
and recover the convergence rate.

\RMIS{}
{Nevertheless, up to this point, we have not yet explained why introducing
new vectors to the search space can reduce the \(v_2\) component of
\(x_1^{(j)}\).
Also, Proposition~\ref{prop:palat} fails to explain the cases when
\(\nex>\nev>1\).
In the next section, we will analyze the full process for a more general case
to answer these questions.}

\RMIS{
Similar conclusions can be obtained when \(\nev>1\).
Let \(V=[v_1,\dotsc,v_n]\) and
\(X^{(j)}=\bigl[x_1^{(j)},x_2^{(j)},\dotsc,x_\nex^{(j)}\bigr]\) with \(\nex\ge\nev>1\).
Suppose \(X^{(j)}\) is written as}{}
\RMISBLOCK{
\[
X^{(j)}=V
\begin{bNiceMatrix}[first-row,last-col]
\nev & \nex-\nev & \\
\Xi_{1,1} & \Xi_{1,2} & ~\nev \\
\Xi_{2,1} & \Xi_{2,2} & ~\nex-\nev \\
\Xi_{3,1} & \Xi_{3,2} & ~n-\nex
\end{bNiceMatrix}.
\]
}{}
\RMIS{
Note that \(x_1^{(j)}\), \(\dotsc\), \(x_\nex^{(j)}\) are arranged in
ascending order with respect to their Ritz values.
Thus, the first \(\nev\) vectors,}{}
\RMISBLOCK{
\begin{equation}
\label{eq:nev}
V\bmat{\Xi_{1,1} \\ \Xi_{2,1} \\ \Xi_{3,1}},
\end{equation}
}{}
\RMIS{
are the approximate solution of the smallest eigenvectors.
Then, by [Theorems A.2 \& A.3]~\cite{LSS2024}, it can be proved that
\(\lVert \Xi_{2,1}\rVert_2\) is a lower-order term compared to
\(\lVert \Xi_{3,1}\rVert_2\).
This means that, when we increase the block size \(\nex\), the components
of~\eqref{eq:nev} will concentrate to larger eigenvalues.
Similar to Proposition~\ref{prop:palat},
in [Theorems A.2 \& A.4]~\cite{LSS2024}, we have proved that the converge is
faster for vectors whose components concentrate less on smaller eigenvectors.
}{}

\section{Detailed convergence analysis of the shrink-and-enlarge technique}
\label{sec:conv}
\ISALL
With the insight we obtained from the small example in
Section~\ref{sec:3dexample}, we are now ready to provide a detailed analysis
for the shrink-and-expand technique on SI in general cases.

To simplify the proof, only in this section, we use \textit{a different set of notation}.
We follow the definition of \(A\) in Proposition~\ref{prop:palat}, and let \(B=A^{-1}\)
and \(\sigma_i=\lambda_i^{-1}\).
Then \(B\in\mathbb C^{n\times n}\) is a Hermitian positive definite matrix with
normalized eigenpairs \((\sigma_1,v_1)\), \((\sigma_2,v_2)\), \(\dotsc\),
\((\sigma_n,v_n)\), where \(\sigma_1\geq\sigma_2\geq\dotsb\geq\sigma_n>0\).
To analyze the inverse iteration \(X\gets A^{-1}X\), we instead analyze the
power iteration \(X\gets BX\).

Suppose that we are performing the subspace iteration algorithm (without
shift-and-invert) to calculate the \(k\) largest eigenpairs of \(B\), and
the search space is spanned by a matrix \(X\in\mathbb C^{n\times k}\),
where \(X\herm X=I_k\).
Since the eigenvectors of \(B\) form an orthogonal basis of
\(\mathbb{C}^n\), \(X\) can be expressed in terms of the eigenvectors of
\(B\) as follows:
\begin{equation}
\label{eq:formx}
X=[V_k,V_{ l\backslash k},V_l^{\perp}]
\begin{bNiceMatrix}[first-row,last-col]
k & \\
X_k & ~k \\
X_{ l\backslash k} & ~l-k \\
X_l^{\perp} & ~n-l
\end{bNiceMatrix},
\end{equation}
where \(V_k=[v_1,\dotsc,v_k]\), \(V_{ l\backslash k}=[v_{k+1},\dotsc,v_{l}]\),
\(V_l^{\perp}=[v_{l+1},\dotsc,v_n]\).%
\footnote{For ease of understanding, readers may interpret \(l\) as the
size of the initial/expanded search space, \(\nex\), and \(k\) as the size
of shrunken subspace \(\nes\). However, it is important to note that the
proof in this section applies to a more general setting.}
For convenience, we further assume \(X_k\in\mathbb C^{k\times k}\) is
nonsingular.%
\footnote{This assumption is plausible in practice, since \(X_k\) converges to
a unitary matrix during the subspace iteration.}

To analyze the effectiveness of the shrink-and-expand technique, we first
establish the following theorem.
\begin{theorem}
\label{thm:rate}
Let \(B\in\mathbb C^{n\times n}\) be a Hermitian positive definite matrix with
normalized eigenpairs \((\sigma_1,v_1)\), \((\sigma_2,v_2)\), \(\dotsc\),
\((\sigma_n,v_n)\).
Assume that \(\sigma_1\geq\sigma_2\geq\dotsb\geq\sigma_n>0\), and \(X\)
is as defined in~\eqref{eq:formx}.
Then the convergence rate of one iteration of subspace iteration
algorithm reads
\begin{equation}
\label{eq:convergerate}
\frac{\tan\angle(V_k,BX)}
{\tan\angle(V_k,X)}
\leq\frac{\sigma_{l+1}}{\sigma_k}
+\frac{\sigma_{k+1}-\sigma_{ l+1}}{\sigma_k}
\frac{\bigl\lVert X_{ l\backslash k}
X_k^{-1}\bigr\rVert_2}
{\left\lVert\bmat{X_{ l\backslash k} \\
X_l^{\perp}}X_k^{-1}\right\rVert_2}.
\end{equation}
\end{theorem}
\begin{proof}
See Appendix~\ref{sec:proof}.
\end{proof}

If the matrix \(X\) is a randomly generated initial guess and \(l\ll n\),
the term
\[
\frac{\bigl\lVert X_{ l\backslash k}
X_k^{-1}\bigr\rVert_2}
{\left\lVert\bmat{X_{ l\backslash k} \\
X_l^{\perp}}X_k^{-1}\right\rVert_2}
\]
is typically much less than \(1\), making the right-hand side
of~\eqref{eq:convergerate} close to \(\sigma_{l+1}/\sigma_k\).
However, during the process of the SI algorithm, this term becomes larger as
\[
\frac{\bigl\lVert X_{ l\backslash k}
X_k^{-1}\bigr\rVert_2}
{\left\lVert\bmat{X_{ l\backslash k} \\
X_l^{\perp}}X_k^{-1}\right\rVert_2}
\rightarrow
\frac{\bigl\lVert\Sigmatwo^mX_{ l\backslash k}
X_k^{-1}\Sigmaone^{-m}\bigr\rVert_2}
{\left\lVert\bmat{\Sigmatwo^m X_{l\backslash k} \\
\Sigmathr^m\Xthr}X_k^{-1}\Sigmaone^{-m}\right\rVert_2}
\rightarrow
1,
\]
where \(m\) is the number of iterations and \(\Sigmaone=\diag\{\sigma_1,\dotsc,\sigma_k\}\),
\(\Sigmatwo=\diag\{\sigma_{k+1},\dotsc,\sigma_l\}\),
\(\Sigmathr=\diag\{\sigma_{l+1},\dotsc,\sigma_n\}\).
This explains why the residual curve of the subspace iteration algorithm is
usually observed to decrease rapidly at the first several iterations, but
goes smoother and finally shows a convergence rate of
\(\sigma_{k+1}/\sigma_k\).

The point of the shrink-and-enlarge technique is that, by expanding
\(X\) with some certain vectors \(\Xexp\), performing the Rayleigh--Ritz
process on \(\Span\{X,\Xexp\}\), and selecting \(k\) vectors from it,
we obtain a new \textit{\(k\)-dimension} search space \(\Span{\Xnew}\).
However, the term
\begin{equation}
\label{eq:newterm}
\frac{\bigl\lVert\Xnew_{ l\backslash k}\Xnew_k^{-1}\bigr\rVert_2}
{\left\lVert\bmat{\Xnew_{l\backslash k}\\
\Xnew_l^{\perp}}\Xnew_k^{-1}\right\rVert_2},
\end{equation}
is significantly smaller.
This helps maintain a high convergence rate even after reducing the subspace.

In the remainder of this section, we present the proof in the
following order:
\begin{enumerate}
\item Firstly, we clarify the structure of the expanded search
space \(\Span\{X,\Xexp\}\).
With some mild assumptions, we can provide an orthogonal basis \(Y\)
with a special structure.
\item Then, with this basis, the Rayleigh--Ritz process performs a
spectral decomposition on \(H=Y\herm BY\), say \(H=C\Theta C\herm\).
We can analyze the structure of \(C\) with the Davis--Kahan theorem.
\item In the end, we estimate the approximate eigenvectors \(YC\)
block by block to derive an estimation of the term~\eqref{eq:newterm}.
\end{enumerate}

\subsection{To find an orthogonal basis of the expanded search space}
Before looking into \(\Span\{X,\Xexp\}\), we first make some assumptions
here.
The \(\Xexp\) introduced should be close to the eigenvectors
\(V_{ l\backslash k}\) and orthogonal.
We represent it as
\begin{equation}
\label{eq:formxexp}
\Xexp
=[V_k,V_{ l\backslash k},V_l^{\perp}]\left(
\begin{bNiceMatrix}[first-row,last-col]
l-k & \\
0 & ~k \\
I_{l-k} & ~l-k \\
0 & ~n-l
\end{bNiceMatrix}
+\epsexp\cdot\Delta\right),
\end{equation}
where \(\lVert\Delta\rVert_2=1\) with
\[
\Delta=\bmat{\Deltaone\\\Deltatwo\\\Deltathr},\quad
\Deltaone\in\mathbb C^{k\times(l-k)},\quad
\Deltatwo\in\mathbb C^{(l-k)\times(l-k)},\quad
\Deltathr\in\mathbb C^{(n-l)\times(l-k)}.
\]
Then an orthogonal basis \(Y\) of \(\Span\{X,\Xexp\}\) can be constructed by
the following theorem.
\begin{theorem}
\label{thm:decomp}
Suppose \(X\) and \(\Xexp\) are of the form~\eqref{eq:formx}
and~\eqref{eq:formxexp}, respectively.
Then there exists an orthogonal basis of \(\Span\{X,\Xexp\}\), named
\(Y\in\mathbb C^{n\times l}\), such that
\(\Span\{Y\}=\Span\{X,\Xexp\}\), \(Y\herm Y=I_l\), and
\begin{equation}
\label{eq:fromy}
Y=V_l+[V_k,V_{ l\backslash k},V_l^{\perp}]
\begin{bNiceMatrix}[first-row,last-col]
k &  l-k & \\
E_{1,1} & E_{1,2} & ~k \\
E_{2,1} & E_{2,2} & ~l-k \\
E_{3,1} & E_{3,2} & ~n-l
\end{bNiceMatrix}
\end{equation}
with \(V_l = [V_k,V_{ l\backslash k}]\).
Moreover, if we define \(\bigl\lVert\Xthr\Xone^{-1}\bigr\rVert_2=\teta\)
and \(\bigl\lVert X\Xone^{-1}\bigr\rVert_2=\heta\), these \(E_{*,*}\)'s can be
bounded by
\[
\begin{aligned}
\lVert E_{1,1}\rVert_2&\leq\frac{1}{2}\teta^2+\epsexp(2\heta+\teta\heta+\frac{1}{2}\teta^2\heta)
+\epsexp^2(4\heta^2+\teta\heta^2)
+3\epsexp^3\heta^2,\\
\lVert E_{2,1}\rVert_2&\leq\epsexp(2\heta+\teta^2\heta)
+\epsexp^2(2\heta^2+2\teta\heta^2)+6\epsexp^3\heta^3,\\
\lVert E_{3,1}\rVert_2&\leq\teta+\frac{1}{2}\teta^3
+\epsexp(\heta+\teta\heta+\frac{3}{2}\teta^2\heta)
+\epsexp^2(\heta^2+4\teta\heta^2)+3\epsexp^3\heta^3,\\
\lVert E_{3,1}\rVert_2&\geq\teta-\frac{1}{2}\teta^3
-\epsexp(\heta+\teta\heta+\frac{3}{2}\teta^2\heta)
-\epsexp^2(\heta^2+4\teta\heta^2)-3\epsexp^3\heta^3,\\
\lVert E_{1,2}\rVert_2&\leq\epsexp,\\
\lVert E_{2,2}\rVert_2&\leq\epsexp,\\
\lVert E_{3,2}\rVert_2&\leq\epsexp.
\end{aligned}
\]
\end{theorem}
\begin{proof}
See Appendix~\ref{sec:proof-decomp}.
\end{proof}
Note that when the subspace iteration algorithm is close to convergence, we
have \(\teta\rightarrow0\) and \(\heta\rightarrow1\).

\subsection{The result of the Rayleigh--Ritz process}
With the expression of \(Y\), we can now perform the Rayleigh--Ritz process.
Since the result of the Rayleigh--Ritz process does not depend on the choice
of orthogonal basis, we assume the approximate eigenvectors are
given by \(YC\), where \(C\) consists of the eigenvectors of the projected
matrix \(H=Y\herm BY\).

To understand the structure of \(C\), we rely on the following perturbation
result as shown in Theorem~\ref{thm:perturbation}.
The proof of this theorem is omitted here since it the same as that
of~\cite[Theorem~3]{LSS2024}.

\begin{theorem}
\label{thm:perturbation}
Let \(H=\Sigma+\delta H\in\mathbb C^{l\times l}\) be a Hermitian matrix with
spectral decomposition \(H=C\Theta C\herm\), where \(\Sigma\) and \(\Theta\)
are real diagonal matrices and \(C\) is unitary.
Partition \(\Sigma\), \(\Theta\) and \(\delta H\) into
\(\Sigma=\diag\set{\Sigma_k,\Sigma_{l\backslash k}}\),
\(\Theta=\diag\set{\Theta_k,\Theta_{l\backslash k}}\), and
\(\delta H=[\delta H_1,\delta H_2]\),
where \(\Sigma_k\), \(\Theta_k\in\mathbb{R}^{k\times k}\), and
\(\delta H_2\in\mathbb{C}^{l\times(l-k)}\).
Suppose
\[
\min\bigl(\spec(\Theta_k)\bigr)>\max\bigl(\spec(\Sigma_{l\backslash k})\bigr)+\alpha,
\]
where \(\alpha>0\) represents the gap between \(\spec(\Sigma_k)\) and
\(\spec(\Theta_{l\backslash k})\).
Then there exist unitary matrices \(C_1\in\mathbb C^{k\times k}\) and
\(C_2\in\mathbb C^{(l-k)\times(l-k)}\) satisfying
\[
C=\bmat{C_1 \\ & C_2}
+\bmat{\delta C_{1,1} & \delta C_{1,2} \\ \delta C_{2,1} & \delta C_{2,2}},
\qquad \lVert\delta C_{i,j}\rVert_2\leq
\begin{cases}
\ceta^2, & i=j,\\
\ceta, & i\neq j,
\end{cases}
\]
in which \(\ceta=\lVert\delta H_2\rVert_2/\alpha\).
\end{theorem}

Applying Theorem~\ref{thm:perturbation} to \(\delta H=Y\herm BY-\Sigma_l\),
and using the expression of \(Y\) given in Theorem~\ref{thm:decomp}, we derive
\[
\begin{aligned}
\lVert\delta H_2\rVert_2
&=\left\lVert(Y\herm BY-\Sigma_l)\bmat{0\\I_{l-k}}\right\rVert_2\\
&=\left\Vert\bmat{E_{2,1}\herm\Sigmatwo\\E_{2,2}\herm\Sigmatwo}
+\bmat{\Sigmaone E_{1,2}\\\Sigmatwo E_{2,2}}
+\bmat{E_{1,1}\herm\Sigmaone E_{1,2}
+E_{2,1}\herm\Sigmatwo E_{2,2}+E_{3,1}\herm\Sigmathr E_{3,2}\\
E_{1,2}\herm\Sigmaone E_{1,2}
+E_{2,2}\herm\Sigmatwo E_{2,2}+E_{3,2}\herm\Sigmathr E_{3,2}}
\right\Vert_2\\
&\le O\bigl(\sigma_1\epsexp(\heta+\teta\heta^2)\bigr).
\end{aligned}
\]
Assuming that \(\epsexp\ll\teta<1\), \(\heta=O(1)\),
\(\sigma_1/\alpha=O(1)\), we obtain
\[
\ceta=\frac{\lVert\delta H_2\rVert_2}{\alpha}=O(\epsexp).
\]
As a result, the eigenvectors of \(H\) satisfy
\begin{equation}
\label{eq:formc}
C=\bmat{C_1 \\ & C_2}
+\bmat{\delta C_{1,1} & \delta C_{1,2} \\ \delta C_{2,1} & \delta C_{2,2}},
\qquad \lVert\delta C_{i,j}\rVert_2=
\begin{cases}
O(\epsexp^2), & i=j,\\
O(\epsexp), & i\neq j,
\end{cases}
\end{equation}
where \(C_1\in\mathbb C^{k\times k}\) and
\(C_2\in\mathbb C^{(l-k)\times(l-k)}\) are unitary.

This result confirms that for a sufficiently small \(\epsexp\), the first \(k\)
and the last \(l-k\) eigenvectors obtained from the Rayleigh--Ritz process
remain, respectively, close to those of the original subspace.

\subsection{To estimate the convergence rate of \(\Xnew\)}
With all the preparation before, we can summarize the following estimation
on the convergence rate of \(\Xnew\).
\begin{theorem}
\label{thm:main}
Let \(B\in\mathbb C^{n\times n}\) be a Hermitian positive definite matrix with
normalized eigenpairs \((\sigma_1,v_1)\), \((\sigma_2,v_2)\), \(\dotsc\),
\((\sigma_n,v_n)\), where \(\sigma_1\geq\sigma_2\geq\dotsb\geq\sigma_n>0\).
The matrices \(X\in\mathbb{C}^{n\times k}\) and
\(\Xexp\in\mathbb{C}^{n\times (l-k)}\) are defined as~\eqref{eq:formx}
and~\eqref{eq:formxexp}, respectively.
Suppose that \(\epsexp\ll\teta<1\), \(\heta=O(1)\), \(\alpha>0\),
\(\sigma_1/\alpha=O(1)\), and \(\epsexp\ll (1-\teta^2)/2\), where
\(\epsexp\), \(\teta\), \(\heta\), and \(\alpha\) are defined
in~\eqref{eq:formxexp}, Theorems~\ref{thm:decomp} and~\ref{thm:perturbation}.

If we perform a Rayleigh--Ritz process on \(\Span\{X,\Xexp\}\) and
select the approximate eigenvectors corresponding to \(k\) largest
eigenvalues to form \(\Xnew\in\mathbb{C}^{n\times k}\), the convergence
rate of the subspace iteration algorithm on \(\Xnew\) is
\begin{equation}
\label{eq:main}
\frac{\tan\angle(V_k,B\Xnew)}
{\tan\angle(V_k,\Xnew)}
\leq\frac{\sigma_{l+1}}{\sigma_k}
+\frac{\sigma_{k+1}-\sigma_{ l+1}}{\sigma_k}
\cdot O(\epsexp).
\end{equation}
\end{theorem}

\begin{proof}
See Appendix~\ref{sec:proof-main}.
\end{proof}

Since \(\epsexp\ll1\), we have proved that \(\Xnew\) can achieve a better
convergence rate even after a shrinkage step.

\begin{remark}
\label{remark:inverseiteration}
The conclusion obtained can be easily represented by \(A\) with the
shift-and-invert version of the SI algorithm.
Suppose we are using Algorithm~\ref{alg:SI} with shift \(\zeta=0\) and
\(\epsexp\), \(\Xnew\), \(V_k\), and \(A\) defined as before, the
convergence rate~\eqref{eq:main} becomes
\[
\frac{\tan\angle(V_k,A^{-1}\Xnew)}
{\tan\angle(V_k,\Xnew)}
\leq\frac{\lambda_k}{\lambda_{l+1}}
+\lambda_k\Bigl(\frac{1}{\lambda_{k+1}}-\frac{1}{\lambda_{l+1}}\Bigr)
\cdot O(\epsexp).
\]
\end{remark}

\begin{remark}
However, extending our proof to the general cases of SD, LOBPCG, and
TraceMIN is not a straightforward task.
We only provide some basic insights here.
\begin{itemize}
\item \textbf{LOBPCG/SD:} To the best of our knowledge, both SD and
LOBPCG share the theoretical worst-case convergence results of PINVIT as
outlined in~\cite{AKNOZ2017}.
It is important to note that PINVIT is equivalent to the SI algorithm with
shift-and-invert when \(A^{-1}\) is used as the preconditioner.
Consequently, the convergence properties for the shift-and-invert variant of
the SI algorithm discussed in Remark~\ref{remark:inverseiteration} are
applicable to PINVIT with the preconditioner~\(A^{-1}\).
For a more general version of PINVIT using a preconditioner that approximates
\(A^{-1}\), it can also be treated as a perturbed SI, and hence potentially
benefit from our shrink-and-expand approach.
\item \textbf{TraceMIN:} As it is shown in~\cite{KSS2013}, that if we use a
sparse direct method to solve the KKT conditions, the TraceMIN algorithm
is mathematically the same as the SI algorithm with shift-and-invert.
Therefore, the shrink-and-expand technique on general TraceMIN can be
understood as an SI with perturbation.
\end{itemize}

\end{remark}

\ISALLEND

\section{Numerical experiments}
\label{sec:numexp}
In this section we will illustrate the efficiency of the shrink-and-expand
technique on four eigensolvers in Section~\ref{sec:prelim}.
All experiments are implemented by MATLAB \RMIS{R2022b}{R2023b} and the
hardware is based on two 16-core Intel Xeon Gold~6226R~2.90~GHz CPUs and
1024~GB of main memory.

\subsection{Experiment settings}
In our experiments, we compute \(\nev\) smallest eigenvalues and the
corresponding eigenvectors of \RMIS{symmetric}{Hermitian positive definite}
matrices listed in Table~\ref{tab:testmatrices}.
These twelve test matrices are chosen from the SuiteSparse Matrix
Collection~\cite{DH2011}, and have been widely used by researchers for testing
large, sparse eigensolvers~\cite{DSYG2018,JZ2022}.
We also remark that some of them, e.g., {\tt{c-65}}, are quite
challenging~\cite{DSYG2018}, providing a more comprehensive assessment of
practical use.
To eliminate the influence of linear solvers, if the test matrix \(A\) is 
\emph{not} positive definite, a shift
\begin{equation}
\label{eq:shift}
\check A\gets A-1.05\cdot\lambda_1I
\end{equation}
will be applied to ensure all of the test matrices are positive
semi-definite, where \(\lambda_1\) represents the \RMIS{}{algebraically}
smallest eigenvalue of \(A\).
For simplicity, we refer \(\check A\) as \(A\) in the rest of this section.
\RMIS{}{Under this setting, all algorithms in Section~\ref{sec:prelim} target
the same set of eigenvalues of \(A\) --- the smallest ones, both algebraically
and in magnitudes.}

\begin{table}[tb!]
\centering
\caption{Information of test matrices.
The scalars \(n\) and \(\nnz(A)\) are the size and the number of nonzero
elements of the matrix, respectively, and the \(\nev\) is the number of
desired \RMIS{eigenparis}{eigenpairs}.}
\begin{tabular}{ccccc}
\hline
No. & Matrix \(A\) & \(n\) & \(\nnz(A)\) & \(\nev\) \\
\hline
1  & {\tt{bcsstm21}}      & \phantom{00}\(3,600\)   & \phantom{0000}\(3,600\)     & \(100\) \\
2  & {\tt{rail\_5177}}    & \phantom{00}\(5,177\)   & \phantom{000}\(35,185\)    & \(100\)  \\
3  & {\tt{Muu}}           & \phantom{00}\(7,102\)   & \phantom{00}\(170,134\)   & \(100\)  \\
4  & {\tt{fv1}}           & \phantom{00}\(9,604\)   & \phantom{000}\(85,264\)    & \(100\)    \\
5  & {\tt{shuttle\_eddy}} & \phantom{0}\(10,429\)  & \phantom{00}\(103,599\)   & \(104\) \\
6  & {\tt{barth5}}        & \phantom{0}\(15,606\)  & \phantom{000}\(61,484\)    & \(156\)   \\
7  & {\tt{Si5H12}}        & \phantom{0}\(19,896\)  & \phantom{00}\(738,598\)   & \(199\)    \\
8  & {\tt{mario001}}      & \phantom{0}\(38,434\)  & \phantom{00}\(204,912\)   & \(384\)   \\
9  & {\tt{c-65}}          & \phantom{0}\(48,066\)  & \phantom{00}\(360,428\)   & \(481\)     \\
10 & {\tt{Andrews}}       & \phantom{0}\(60,000\)  & \phantom{00}\(760,154\)   & \(500\)       \\
11 & {\tt{Ga3As3H12}}     & \phantom{0}\(61,349\)  & \(5,970,947\) & \(500\)     \\
12 & {\tt{Ga10As10H30}}   & \(113,081\)            & \(6,115,633\) & \(500\)                   \\
\hline
\end{tabular}
\label{tab:testmatrices}
\end{table}

In our experiments, we employ the convergence criterion for a single
eigenpair as
\begin{equation}
\label{eq:conv}
\lVert A\hat v_i-\hat v_i\hat\lambda_i\rVert_2\le
{\tt{tol}}\cdot\bigl(\lVert A\rVert_2\lVert \hat v_i\rVert_2+\lVert \hat v_i\rVert_2\lvert\hat \lambda_i\rvert\bigr),
\end{equation}
where \((\hat\lambda_i,\hat v_i)\) is the computed approximate eigenpair,
and the threshold \({\tt{tol}}\) is set to \(10^{-10}\) for all algorithms.
The overall relative residual of the \(j\)th iteration, \(r^{(j)}\), is
defined as
\begin{equation}
\label{eq:rk}
r^{(j)}=\max_{i=1,\dotsc,\nev}
\RMIS{
\frac{\lVert A\rVert_2\lVert \hat v_i\rVert_2+\lVert\hat v_i\rVert_2\lvert\hat\lambda_i\rvert}
{\lVert A\hat v_i-\hat v_i\hat\lambda_i\rVert_2}}
{\frac{\lVert A\hat v_i-\hat v_i\hat\lambda_i\rVert_2}
{\lVert A\rVert_2\lVert \hat v_i\rVert_2+\lVert\hat v_i\rVert_2\lvert\hat\lambda_i\rvert}}
,
\end{equation}
where \(\hat\lambda_1\), \(\hat\lambda_2\), \(\dotsc\) are sorted in ascending order
with respect to their magnitude.
Note that \(r^{(j)}\le{\tt{tol}}\) indicates that all the \(\nev\) smallest
eigenpairs satisfy the convergence criterion~\eqref{eq:conv}, which can be
regarded as a signal of convergence.
The residual~\eqref{eq:rk} will be used for estimating the convergence rate
in~\eqref{eq:ckslope} and~\eqref{eq:ckslopek}, and for plotting the
convergence history.

\RMIS{As noted in Section~\ref{subsec:strategies}, we impose a few warm-up
iterations before applying the first shrinkage.
Instead of a fixed number of iterations, we use an adaptive way to determine
the timing of the first shrinkage.
In our experiments, the first shrinkage will not be employed until the
minimum iteration \(j\ge5\) and maximum residual \(r^{(j)}\le10^{-4}\) are
reached.
}{}

\subsection{Overall performance}
In this subsection, comparing with the algorithms without using the
shrink-and-expand technique, we employ Algorithms~\ref{alg:SI}, \ref{alg:SD},
\ref{alg:LOBPCG}, and~\ref{alg:TM} with the three different shrink-and-expand
strategies discussed in Section~\ref{subsec:strategies} to compute the
eigenpairs of the test matrices outlined in Table~\ref{tab:testmatrices}.
\RMIS{}{Locking techniques are not equipped in these algorithms, with the
exception that LOBPCG employs the soft locking technique~\cite{DSYG2018}.
Furthermore, for fairness, we do not apply any preconditioning in SD and
LOBPCG (i.e., \(T=I\)).}

In all of our test examples, we compute \(1\%\) smallest (magnitude)
eigenpairs of test matrices, at most \(500\), at least \(100\).
When computing \(\nev\) smallest eigenpairs, we set the initial block size
\(\nex=1.5\cdot\nev\) for the LOBPCG algorithm and \(\nex=2\cdot\nev\) for
other algorithms, and set \(\nes=\nev+5\).%
\RMIS{}{\footnote{These choices are consistent with commonly used parameters
for these eigensolvers; see, e.g.,~\cite{KSS2013, KMS2023, SMU2019}.}}

\RMIS{
To enroll the three strategies in Algrirthm~\ref{alg:fix} and~\ref{alg:slope},
we set the period of taking expansion \(\je=12\), the number of iterations of
taking shrinkage after the last expansion \(\js=2\), and the threshold of the
expansion \(\mu=1.1\).}
{When enrolling the {\tt{slope}} strategies (Algorithm~\ref{alg:fix}), we set
the period of taking expansion \(\je=12\) and the number of iterations of
taking shrinkage after the last expansion \(\js=2\).
And for the {\tt{slope}} and {\tt{slopek}} strategies
(Algorithm~\ref{alg:slope}), we set the threshold of the expansion \(\mu=1.1\)
and the period for taking average \(\jp=10\).
The warming up parameters are set to \(\jwarm=5\) and \(\rwarm=10^{-4}\).}

The test results (runtime, total iterations, and acceleration rate) of
Algorithms~\ref{alg:SI}, \ref{alg:SD}, \ref{alg:LOBPCG}, and~\ref{alg:TM} can
be found in Tables~\ref{tab:SIexps}, \ref{tab:SDexps}, \ref{tab:LOBPCGexps},
and~\ref{tab:TMexps}, respectively.
To avoid abuse of computational resources, we impose a limit of \(3600\)
seconds;
any test that exceeds this time limit is terminated and marked with
``\(\infty\)'' in the tables.

For most cases, our technique can achieve an acceleration of \(20\%\) to
\(40\%\).
From the perspective of iterations, the application of the shrink-and-expand
technique results in an increase in the number of iterations by around
\(10\%\).
However, due to the cheaper cost per iteration, the overall runtime is
effectively reduced.

There are also some cases where the performance of our technique is not that
satisfactory.
For instance, when employing our technique to the SD algorithm to solve
{\tt{c-65}}, it results in a significant increase in the number of iterations;
see Table~\ref{tab:SDexps}.
This is because the convergence history is observed to be very sensitive to
block size.

The improvement brought about by the shrink-and-expand technique varies for
different algorithms as well.
Among all these algorithms, the LOBPCG algorithm shows the best performance
gain with our technique.
Typically, the LOBPCG algorithm spends a significant amount of time on
orthogonalization, which can be efficiently reduced by decreasing the block
size.
Additionally, the residual curve of some algorithms, such as the SD algorithm,
is fluctuating, which presents challenges for applying the shrinkage.
It can be seen in Table~\ref{tab:SDexps} that the efficiency varies for
different examples.

\RMIS{}
{Another problem that needs to be discussed here is whether the four algorithms still converge to the desired
eigenvalues after employing the shrink-and-expand technique.
In our experiments, this is true for SI, SD, and LOBPCG.
However, for TraceMIN, we observe that, regardless of whether the shrink-and-expand technique is used, a
small fraction (less than \(1\%\)) of the computed eigenvalues converged to slightly
larger ones.
This is likely caused by the 
employment of the inexact linear solver, which makes TraceMIN no longer
mathematically converge to the smallest magnitude eigenpairs.}

\begin{table}[tb!]
\centering
\caption{Runtime (in seconds) and iterations by the SI algorithm with the shrink-and-expand technique.}
\ifsimax
\scriptsize
\else
\small
\fi
\begin{NiceTabular}{c|cc|ccc|ccc|ccc}
\hline
& \multicolumn{2}{c}{w/o} & \multicolumn{3}{c}{{\tt{fix}}} & \multicolumn{3}{c}{{\tt{slope}}} & \multicolumn{3}{c}{{\tt{slopek}}}\\
No. & time & iter & time & iter & save & time & iter & save & time & iter & save\\
\hline
1  & \(0.8064\) & \(30 \) & \(0.5921\) & \(30 \) & \(27 \%\) & \(0.5249\) & \(30 \) & \(35 \%\) & \(0.5245\) & \(30 \) & \(35 \%\)   \\
2  & \(7.2   \) & \(125\) & \(5.522 \) & \(131\) & \(23 \%\) & \(5.239 \) & \(136\) & \(27 \%\) & \(5.037 \) & \(137\) & \(30 \%\)   \\
3  & \(5.445 \) & \(76 \) & \(4.636 \) & \(82 \) & \(15 \%\) & \(4.411 \) & \(84 \) & \(19 \%\) & \(4.423 \) & \(86 \) & \(19 \%\)   \\
4  & \(15.49 \) & \(113\) & \(11.22 \) & \(119\) & \(28 \%\) & \(11.32 \) & \(121\) & \(27 \%\) & \(11.17 \) & \(123\) & \(28 \%\)   \\
5  & \(26.43 \) & \(195\) & \(18.86 \) & \(205\) & \(29 \%\) & \(18.65 \) & \(216\) & \(29 \%\) & \(17.5  \) & \(219\) & \(34 \%\)   \\
6  & \(40.19 \) & \(115\) & \(28.6  \) & \(125\) & \(29 \%\) & \(27.34 \) & \(130\) & \(32 \%\) & \(26.42 \) & \(135\) & \(34 \%\)   \\
7  & \(161.4 \) & \(58 \)  & \(112   \) & \(61 \)  & \(31 \%\) & \(106.7 \) & \(62 \)  & \(34 \%\) & \(105.8 \) & \(63 \)  & \(34 \%\)  \\
8  & \(278.3 \) & \(114\) & \(194.8 \) & \(122\) & \(30 \%\) & \(190.3 \) & \(126\) & \(32 \%\) & \(188.3 \) & \(130\) & \(32 \%\)  \\
9  & \(2220  \) & \(538\) & \(2532  \) & \(1021\) & \(-14 \%\) & \(2692  \) & \(1148\) & \(-21 \%\) & \(2710  \) & \(1200\) & \(-22 \%\) \\
10 & \(1982  \) & \(79 \) & \(1342  \) & \(83 \)  & \(32 \%\) & \(1320  \) & \(85 \)  & \(33 \%\) & \(1328  \) & \(86 \)  & \(33 \%\)  \\
11 & \(1760  \) & \(46 \) & \(1269  \) & \(55 \)  & \(28 \%\) & \(1260  \) & \(52 \)  & \(28 \%\) & \(1362  \) & \(60 \)  & \(23 \%\)  \\
12& \(\infty\)   & -    & \(\infty\)  & -    & -    & \(\infty\)   & -    & -     & \(\infty\)   & -    & -         \\
\hline
\end{NiceTabular}\\
\label{tab:SIexps}
\end{table}

\begin{table}[tb!]
\centering
\caption{Runtime (in seconds) and iterations by the SD algorithm with the
shrink-and-expand technique.}
\ifsimax
\scriptsize
\else
\small
\fi
\begin{NiceTabular}{c|cc|ccc|ccc|ccc}
\hline
& \multicolumn{2}{c}{w/o} & \multicolumn{3}{c}{{\tt{fix}}} & \multicolumn{3}{c}{{\tt{slope}}} & \multicolumn{3}{c}{{\tt{slopek}}}\\
No. & time & iter & time & iter & save & time & iter & save & time & iter & save\\
\hline
1  & \(0.5407\)  & \(12 \) & \(0.349  \) & \(10 \) & \(35 \%\)  & \(0.4894 \) & \(12 \) & \(9   \%\)  & \(0.3708 \) & \(10 \) & \(31 \%\)  \\
2  & \(16.04 \)  & \(235\) & \(12.76  \) & \(247\) & \(20 \%\)  & \(15.64  \) & \(245\) & \(2   \%\)  & \(11.99  \) & \(249\) & \(25 \%\)  \\
3  & \(76.38 \)  & \(981\) & \(56.33  \) & \(1077\)& \(26 \%\)  & \(73.81  \) & \(1069\)& \(3   \%\)  & \(55.63  \) & \(1082\)& \(27 \%\)  \\
4  & \(21.83 \)  & \(274\) & \(22.59  \) & \(293 \)& \(-3 \%\)  & \(25.33  \) & \(289 \)& \(-16 \%\)  & \(21.99  \) & \(297 \)& \(-1 \%\)  \\
5  & \(131.4 \)  & \(527\) & \(82.28  \) & \(594 \)& \(37 \%\)  & \(150.7  \) & \(578 \)& \(-15 \%\)  & \(82.14  \) & \(600 \)& \(37 \%\)  \\
6  & \(936.3 \)  & \(1319\)& \(516.8  \) & \(1429\)& \(45 \%\)  & \(703    \) & \(1378\)& \(25  \%\)  & \(486.9  \) & \(1449\)& \(48 \%\)  \\
7  & \(358.6 \)  & \(360 \)& \(315.3  \) & \(384 \)& \(12 \%\)  & \(328.7  \) & \(379 \)& \(8   \%\)  & \(317.5  \) & \(391 \)& \(11 \%\)  \\
8  & \(3641  \)  & \(1143\)& \(2264   \) & \(1226\)& \(38 \%\)  & \(2798   \) & \(1208\)& \(23  \%\)  & \(2113   \) & \(1249\)& \(42 \%\)  \\
9  & \(2037  \)  & \(462 \)& \(1756   \) & \(550 \)& \(14 \%\)  & \(2069   \) & \(555 \)& \(-2  \%\)  & \(1736   \) & \(574 \)& \(15 \%\)  \\
10 & \(1644  \)  & \(273 \)& \(1179   \) & \(293 \)& \(28 \%\)  & \(1400   \) & \(289 \)& \(15  \%\)  & \(1093   \) & \(302 \)& \(34 \%\)  \\
11 & \(2481  \)  & \(379 \)& \(1699   \) & \(435 \)& \(32 \%\)  & \(2121   \) & \(408 \)& \(15  \%\)  & \(1623   \) & \(468 \)& \(35 \%\)  \\
12& \(\infty\)   & -    & \(\infty\)  & -    & -    & \(\infty\)   & -    & -     & \(\infty\)   & -    & -     \\
\hline
\end{NiceTabular}
\label{tab:SDexps}
\end{table}

\begin{table}[tb!]
\centering
\caption{Runtime (in seconds) and iterations by the LOBPCG algorithm with the
shrink-and-expand technique.
No shrink-and-expand technique was employed in {\tt{bcsstm21}},
since the number of iterations did not reach the threshold.}
\label{tab:LOBPCGexps}
\ifsimax
\scriptsize
\else
\small
\fi
\begin{NiceTabular}{c|cc|ccc|ccc|ccc}
\hline
& \multicolumn{2}{c}{w/o} & \multicolumn{3}{c}{{\tt{fix}}} & \multicolumn{3}{c}{{\tt{slope}}} & \multicolumn{3}{c}{{\tt{slopek}}}\\
No. & time & iter & time & iter & save & time & iter & save & time & iter & save\\
\hline
1  & \(0.1841\)  & \(3  \) & \(0.1629 \) & \(3  \) & \(12 \%\)  & \(0.2166 \) & \(3  \) & \(-18 \%\) & \(0.1709 \) & \(3  \) & \(7  \%\)   \\
2  & \(6.723 \)  & \(79 \) & \(5.469  \) & \(80 \) & \(19 \%\)  & \(6.12   \) & \(79 \) & \(9   \%\)  & \(5.493  \) & \(79 \) & \(18 \%\)  \\
3  & \(19.27 \)  & \(186\) & \(13.31  \) & \(166\) & \(31 \%\)  & \(15.93  \) & \(167\) & \(17  \%\)  & \(13.58  \) & \(165\) & \(30 \%\)  \\
4  & \(11.69 \)  & \(73 \) & \(8.916  \) & \(78 \) & \(24 \%\)  & \(9.424  \) & \(77 \) & \(19  \%\)  & \(9.138  \) & \(82 \) & \(22 \%\)  \\
5  & \(30.37 \)  & \(159\) & \(20.64  \) & \(146\) & \(32 \%\)  & \(30.37  \) & \(143\) & \(0   \%\)  & \(27.57  \) & \(149\) & \(9  \%\)  \\
6  & \(117.3 \)  & \(207\) & \(73.81  \) & \(191\) & \(37 \%\)  & \(77.14  \) & \(189\) & \(34  \%\)  & \(66.43  \) & \(187\) & \(43 \%\)  \\
7  & \(82.36 \)  & \(91 \) & \(76.35  \) & \(97 \) & \(7  \%\)  & \(82.41  \) & \(95 \) & \(0   \%\)  & \(77.74  \) & \(97 \) & \(6  \%\)  \\
8  & \(832.1 \)  & \(206\) & \(565.1  \) & \(189\) & \(32 \%\)  & \(625.8  \) & \(190\) & \(25  \%\)  & \(558.2  \) & \(191\) & \(33 \%\)  \\
9  & \(1468  \)  & \(184\) & \(849.5  \) & \(149\) & \(42 \%\)  & \(909.5  \) & \(137\) & \(38  \%\)  & \(859.4  \) & \(158\) & \(41 \%\)  \\
10 & \(604   \)  & \(90 \) & \(436.3  \) & \(91 \) & \(28 \%\)  & \(506.3  \) & \(93 \) & \(16  \%\)  & \(430.7  \) & \(93 \) & \(29 \%\)  \\
11 & \(746.1 \)  & \(103\) & \(565.4  \) & \(100\) & \(24 \%\)  & \(615.6  \) & \(102\) & \(17  \%\)  & \(605.8  \) & \(108\) & \(19 \%\)  \\
12 & \(1955  \)  & \(147\) & \(1411   \) & \(137\) & \(28 \%\)  & \(1651   \) & \(145\) & \(16  \%\)  & \(1611   \) & \(154\) & \(18 \%\)  \\
\hline
\end{NiceTabular}
\end{table}

\begin{table}[tb!]
\centering
\caption{Runtime (in seconds) and iterations by the TraceMIN algorithm with
the shrink-and-expand technique.
For {\tt{mario001}}, the accelerate rate is marked as \(100\%\) because
the TraceMIN algorithm without the shrink-and-expand technique failed to
converge in \(3600\) seconds.}
\ifsimax
\scriptsize
\else
\small
\fi
\begin{NiceTabular}{c|cc|ccc|ccc|ccc}
\hline
& \multicolumn{2}{c}{w/o} & \multicolumn{3}{c}{{\tt{fix}}} & \multicolumn{3}{c}{{\tt{slope}}} & \multicolumn{3}{c}{{\tt{slopek}}}\\
No. & time & iter & time & iter & save & time & iter & save & time & iter & save\\
\hline
1  & \(3.897\)  & \(30 \) & \(2.39  \) & \(30 \) & \(39 \%\) & \(2.272 \) & \(30 \) & \(42 \%\) & \(2.251 \) & \(30 \) & \(42 \%\)   \\
2  & \(45.21\)  & \(123\) & \(34.25 \) & \(129\) & \(24 \%\) & \(33.6  \) & \(132\) & \(26 \%\) & \(33.34 \) & \(134\) & \(26 \%\)   \\
3  & \(46.97\)  & \(83 \) & \(35.89 \) & \(89 \) & \(24 \%\) & \(34.73 \) & \(89 \) & \(26 \%\) & \(35.33 \) & \(93 \) & \(25 \%\)   \\
4  & \(67.19\)  & \(111\) & \(48.61 \) & \(116\) & \(28 \%\) & \(48.19 \) & \(118\) & \(28 \%\) & \(47.36 \) & \(118\) & \(30 \%\)   \\
5  & \(135.9\)  & \(197\) & \(96.22 \) & \(208\) & \(29 \%\) & \(91.42 \) & \(218\) & \(33 \%\) & \(91.42 \) & \(222\) & \(33 \%\)   \\
6  & \(238.7\)  & \(119\) & \(150   \) & \(131\) & \(37 \%\) & \(143.5  \) & \(133\) & \(40 \%\) & \(141.2  \) & \(138\) & \(41 \%\)  \\
7  & \(450.9\)  & \(59 \) & \(302.3 \) & \(62 \) & \(33 \%\) & \(299.2 \) & \(63 \) & \(34 \%\) & \(296.8 \) & \(64 \) & \(34 \%\)   \\
8  & \(\infty\)  & -      & \(2289  \) & \(125\) & \(100\%\) & \(2220  \) & \(131\) & \(100\%\) & \(2233  \) & \(134\) & \(100\%\)  \\
9 &  \(\infty\)  & -   & \(\infty\)  & -   & -    & \(\infty\)  & -   & -    & \(\infty\)  & -   & -           \\
10&  \(\infty\)  & -   & \(\infty\)  & -   & -    & \(\infty\)  & -   & -    & \(\infty\)  & -   & -                \\
11&  \(\infty\)  & -   & \(\infty\)  & -   & -    & \(\infty\)  & -   & -    & \(\infty\)  & -   & -                 \\
12&  \(\infty\)  & -   & \(\infty\)  & -   & -    & \(\infty\)  & -   & -    & \(\infty\)  & -   & -                 \\
\hline
\end{NiceTabular}
\label{tab:TMexps}
\end{table}

\subsection{Performance of different strategies}
\label{sub-sec:numexp-3str}
We have introduced and employed three different strategies to determine the
timing of applying the shrink-and-expand technique.
Here, we will discuss the properties of each of these three strategies based
on the results of Tables~\ref{tab:SIexps}, \ref{tab:SDexps},
\ref{tab:LOBPCGexps}, and~\ref{tab:TMexps}.

We begin with the cases where the decrease of the residual is relatively
stable.
In Figure~\ref{fig:LOBPCG3str}, we present the convergence history of solving
the \RMIS{eigenparis}{eigenpairs} of {\tt{Muu}} by the LOBPCG algorithm, whose results are also
shown in Table~\ref{tab:LOBPCGexps}.
Observing from the figure, the {\tt{fix}} strategy, of course, exhibits a
consistent frequency in applying both shrinkage and expansion.
In contrast, the {\tt{slope}} strategy appears to be more sensitive,
particularly during the last half of the iterations.
The {\tt{slopek}} strategy tends to employ expansion less frequently,
with only four instances throughout the entire process.
Nevertheless, from the aspect of efficiency, the performance differences among
the three strategies are not significant.

\begin{figure}
\centering
\ifsimax
\includegraphics[height=5cm]{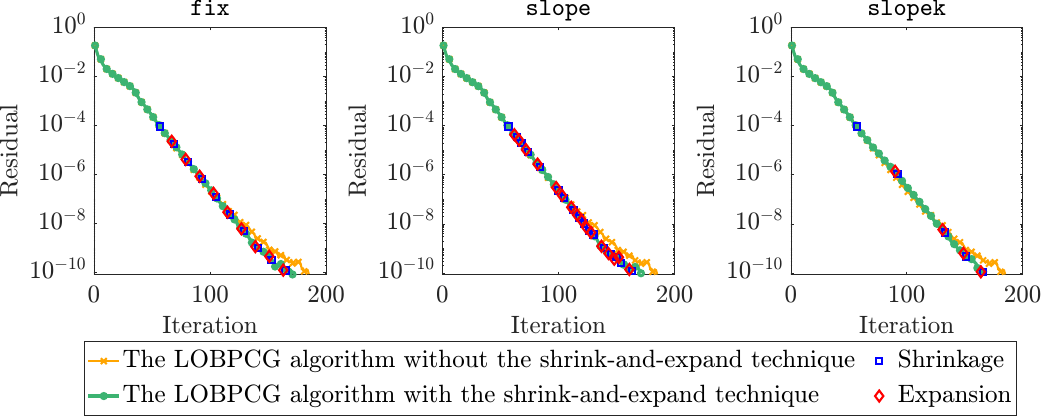}
\else
\includegraphics[height=5.8cm]{fig/LOBPCG3str.pdf}
\fi
\caption{Using the LOBPCG algorithm to compute \(100\) smallest
eigenpairs of the (shifted) {\tt{Muu}} matrix.
Three different shrinkage and expansion strategies---{\tt{fix}},
{\tt{slope}}, and {\tt{slopek}} are applied, respectively.
}
\label{fig:LOBPCG3str}
\end{figure}

However, if the convergence history is more fluctuating, the outcomes are
quite different.
In Figure~\ref{fig:SD3str}, we apply the SD algorithm on the same matrix,
{\tt{Muu}}.
Here, the residual curve experiences significant fluctuations, making it
difficult for the {\tt{slope}} strategy to appropriately determine the
timing for employing the expansion.
More specifically, an expansion will be applied whenever there is an
increase on the residual curve, since \RMIS{\(c_j<0<\mu\cmax\)}{\(c^{(j)}<0<\cmax\)};
see Figure~\ref{fig:SD3str} (bottom middle).
However, this is not necessary.
Even if the residual curve fluctuates all the time, it does not imply a
decay in the overall convergence rate.
In contrast, the overall trend remains stable.
In such cases, {\tt{slopek}} can avoid being affected by taking the average
slope of the residual curve.
From the third row of Table~\ref{tab:SDexps}, it can be found that the
{\tt{slopek}} strategy applies the shrinkage and the expansion more wisely,
effectively harnessing the potential of the algorithm, and achieving a
noticeably faster convergence.

\begin{figure}[!tb]
\centering
\ifsimax
\includegraphics[height=8cm]{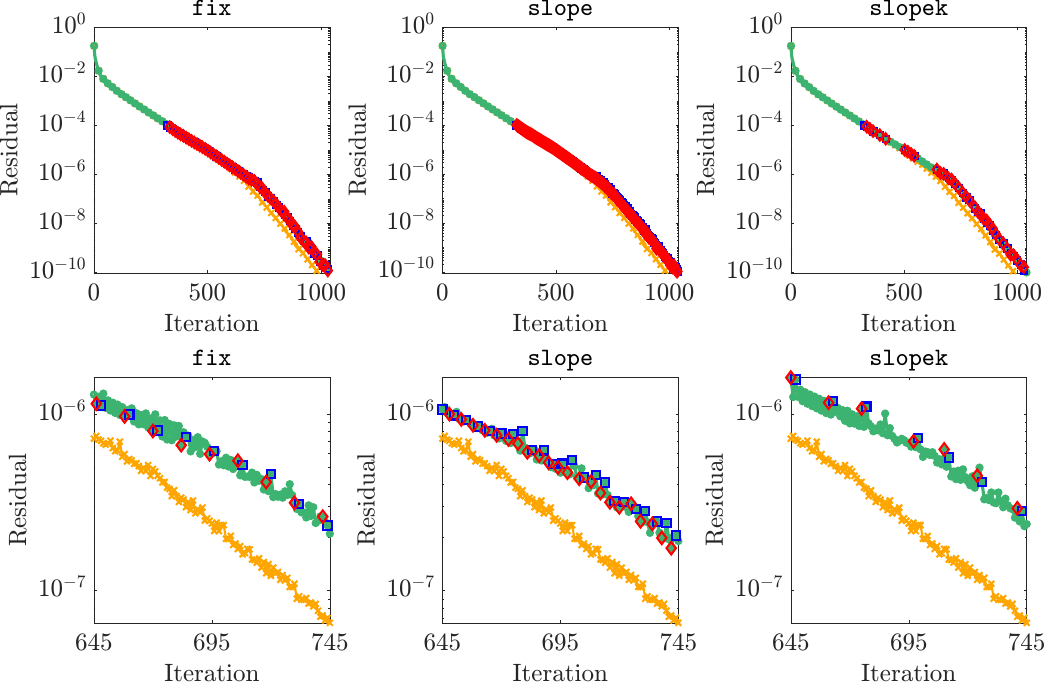}\\
\quad\\
\includegraphics[height=1cm]{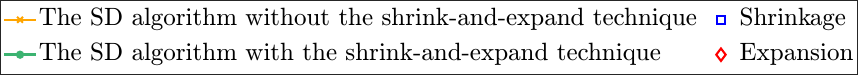}
\else
\includegraphics[height=9.5cm]{fig/SD3str1.pdf}\\
\quad\\
\includegraphics[height=1.2cm]{fig/SD3str2.pdf}
\fi
\caption{Using the SD algorithm to compute \(100\) smallest
eigenpairs of the (shifted) {\tt{Muu}} matrix.
Three different shrink-and-expand strategies---{\tt{fix}},
{\tt{slope}}, and {\tt{slopek}} are applied, respectively.
Both the shrinkage and the expansion are very dense in the {\tt{slope}} strategy,
making it hard to distinguish the curve.
Therefore, we have magnified the curve between iterations \RMIS{\(575\)--\(625\)}
{\(645\)--\(745\)} in the lower part to better illustrate the convergence history.}
\label{fig:SD3str}
\end{figure}

\subsection{Comparison on different block sizes}
\label{sub-sec:numexp-block}
To demonstrate that the acceleration of our shrink-and-expand technique is
not simply due to the change in block size, we perform the LOBPCG algorithm on
all test matrices with different block sizes.
The runtime and the number of iterations are listed in Table~\ref{tab:nexs}.
While the optimal block size depends on the spectral gap of the specific
matrices, \(1.5\cdot\nev\) is a reasonable choice, which achieves the fastest
convergence for most of the test examples.
Nevertheless, the performance is further enhanced when employing the
shrink-and-expand technique, which surpasses the performance of the algorithm
with the optimal block size alone on all examples except for {\tt{bcsstm21}}
(where no shrink-and-expand technique employed) and {\tt{shuttle\_eddy}}.

\begin{table}[tb!]
\centering
\caption{Runtime (in seconds) and iterations by the LOBPCG algorithm with
different block size.
Here, the block size is \(\nex=\alpha\cdot\nev\), where
\(\alpha\in\set{1,1.5,2,3,5}\);
for completeness, the result of \(\nex=\nes=\nev+5\) is also listed.
The LOBPCG algorithm with the shrink-and-expand technique and {\tt{fix}}
strategy is also attached at the last column, and is marked as `with SE'.
The maximum iteration is set to \(1000\), and the time limit is set to \(3600\)
seconds.
The examples exceed these limits are marked with \(\infty\).
We highlight the fastest runtime (without shrink-and-expand) of each matrix
in boldface.}
\ifsimax
\tiny
\else
\scriptsize
\fi
\begin{NiceTabular}{c|r@{\hspace{0.2em}}lr@{\hspace{0.2em}}lr@{\hspace{0.2em}}lr@{\hspace{0.2em}}lr@{\hspace{0.2em}}lr@{\hspace{0.2em}}l|r@{\hspace{0.2em}}l}
\hline
& \multicolumn{12}{c}{w/o SE} & \multicolumn{2}{c}{with SE}\\
No. & \multicolumn{2}{c}{\(\nev\)} &\multicolumn{2}{c}{\(\nev+5\) (\(=\nes\))} & \multicolumn{2}{c}{\(1.5\cdot\nev\)} & \multicolumn{2}{c}{\(2\cdot\nev\)} & \multicolumn{2}{c}{\(3\cdot\nev\)} & \multicolumn{2}{c}{\(5\cdot\nev\)} & \multicolumn{2}{c}{\(1.5\cdot\nev\)} \\
\hline
1&\(0.1157\)&\((3)\)&\(\textbf{0.1076}\)&\((3)\)&\(0.1577\)&\((3)\)&\(0.2002\)&\((3)\)&\(0.4137\)&\((3)\)&\(0.9523\)&\((3)\)  & \(0.1629 \) & \((3  )\)\\
2&\(14.09\)&\((615)\)&\(7.979\)&\((207)\)&\(\textbf{6.526}\)&\((74)\)&\(6.579\)&\((53)\)&\(10.04\)&\((39)\)&\(15.39\)&\((28)\)  & \(5.469  \) & \((80 )\)\\
3&\(19.02\)&\((482)\)&\(\textbf{18.3}\)&\((381)\)&\(18.88\)&\((186)\)&\(28.2\)&\((123)\)&\(44.99\)&\((89)\)&\(49.51\)&\((58)\)  & \(13.31  \) & \((166)\)\\
4&\(22.61\)&\((941)\)&\(14.45\)&\((352)\)&\(\textbf{11.79}\)&\((73)\)&\(21.73\)&\((53)\)&\(24.65\)&\((38)\)&\(33.01\)&\((28)\)  & \(8.916  \) & \((78 )\)\\
5&\(16.84\)&\((530)\)&\(\textbf{12.87}\)&\((239)\)&\(37.3\)&\((159)\)&\(42.73\)&\((106)\)&\(50.88\)&\((76)\)&\(63.83\)&\((52)\)  & \(20.64  \) & \((146)\)\\
6&\(-\)&\((\infty)\)&\(96.13\)&\((686)\)&\(\textbf{115.1}\)&\((212)\)&\(141.4\)&\((133)\)&\(149.4\)&\((94)\)&\(177.6\)&\((67)\)  & \(73.81  \) & \((191)\)\\
7&\(135.1\)&\((734)\)&\(135.8\)&\((273)\)&\(\textbf{84.36}\)&\((92)\)&\(93.41\)&\((67)\)&\(102\)&\((49)\)&\(145.3\)&\((36)\)  & \(76.35  \) & \((97 )\)\\
8&\(-\)&\((\infty)\)&\(-\)&\((\infty)\)&\(836.7\)&\((208)\)&\(800.1\)&\((137)\)&\(\textbf{793.6}\)&\((89)\)&\(1223\)&\((63)\)  & \(565.1  \) & \((189)\)\\
9&\(\textbf{940.3}\)&\((175)\)&\(950.6\)&\((175)\)&\(1147\)&\((137)\)&\(1461\)&\((125)\)&\(1866\)&\((96)\)&\(2952\)&\((68)\)  & \(849.5  \) & \((149)\)\\
10&\(-\)&\((\infty)\)&\(1224\)&\((707)\)&\(\textbf{603.1}\)&\((90)\)&\(906.2\)&\((64)\)&\(976.2\)&\((47)\)&\(1603\)&\((37)\)  & \(436.3  \) & \((91 )\)\\
11&\(\infty\)&\((-)\)&\(1378\)&\((572)\)&\(\textbf{763.2}\)&\((103)\)&\(1857\)&\((87)\)&\(2006\)&\((72)\)&\(2990\)&\((57)\)  & \(565.4  \) & \((100)\)\\
12&\(\infty\)&\((-)\)&\(3281\)&\((640)\)&\(\textbf{1997}\)&\((147)\)&\(\infty\)&\((-)\)&\(\infty\)&\((-)\)&\(\infty\)&\((-)\)  & \(1411   \) & \((137)\)\\
\hline
\end{NiceTabular}
\label{tab:nexs}
\end{table}

Furthermore, it also stresses the fact that, as we stated in
Section~\ref{sec:intro}, while increasing the block size can improve the
convergence rate, it may not necessarily decrease the total execution time.
In fact, the runtime decreases initially with increasing block size but then
starts to rise.
This is because increasing the block sizes results in a significant increase
in computational complexity, while the improvement in the convergence rate
may not be as rapid.
Moreover, for large-scale problems, the increased memory requirements of
larger block sizes are often unaffordable.
In contrast, our shrink-and-expand technique can achieve acceleration without
additional memory usage, making it advantageous for handling large-scale
problems.

\subsection{The vectors appended in the expansion}
\label{sub-sec:exp-newvec}
In the expansion, we have to append several vectors to expand the search
subspace.
In our implementation, the vectors discarded in the last shrinkage are reused
here.
A natural question is whether such a choice is efficient.
In Figure~\ref{fig:vectortype}, we illustrate the convergence history of SI
and LOBPCG, both with the {\tt{fix}} strategy, by appending different types of
vectors in the expansion.
We use three different types of vectors for comparison.
The `\(\Xdrop\)', as already illustrated in Figure~\ref{fig:flowchart}, uses
the vectors discarded in the last shrinkage.
The `Random' uses a randomly generated, normally distributed matrix.
And for the SI algorithm, it is also reasonable to update the saved vectors
\(\Xdrop\) by \(A^{-1}\Xdrop\), but without orthogonalization in each
iteration.
We name this method `Powered \(\Xdrop\)'.

\begin{figure}[!tb]
\centering
\ifsimax
\includegraphics[height=6.5cm]{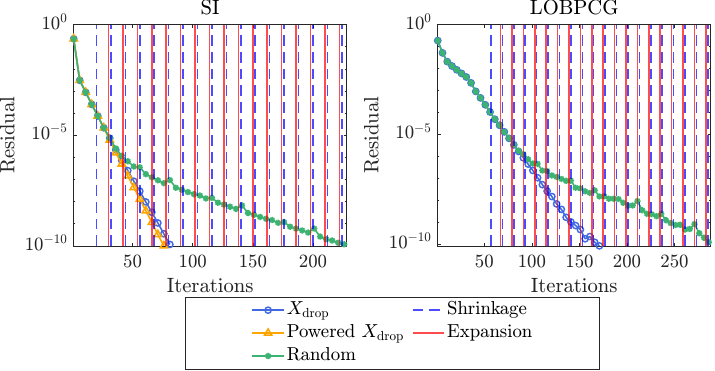}
\else
\includegraphics[height=7cm]{VectorTypes}
\fi
\caption{Use the SI (Left) algorithm and the LOBPCG (right) algorithm with
the shrink-and-expand technique to compute 100 smallest
eigenpairs of the (shifted) {\tt{Muu}} matrix.
Three different types of vectors are used in the expansion.
For `\(\Xdrop\)', the vectors dropped in the last shrinkage are used;
for `Powered \(\Xdrop\)', we use the same vectors as \(\Xdrop\),
but power it by \(A^{-1}\Xdrop\) in each iteration (without orthogonlization);
for `Random', randomly generated normal distributed matrices are used.
In both examples, we use the {\tt{fix}} strategy.
Since the timing of applying the shrink-and-expand technique are the same
for each curve, we use dashed line to represent the shrinkage and the
expansion.}
\label{fig:vectortype}
\end{figure}

\RMIS{It can be observed from Figure~\ref{fig:vectortype} that the reused vectors
perform apparently better than the randomly generated ones.
As for `Powered \(\Xdrop\)', its performance is slightly better than the
`\(\Xdrop\)'.
This experiment also reveals that the vectors appended in the
expansion can have a significant impact on the convergence of our
shrink-and-expand technique.}
{This confirms that, as proved in Section~\ref{sec:conv}, the performance
of the shrink-and-expand technique depends on how accurately the
expanded vectors approximate the eigenvectors \(u_{\nes+1}\), \(\dotsc\),
\(u_{\nex}\).
Since, in the `\(\Xdrop\)' case, the expanded vectors are already the
approximations of \(u_{\nes+1}\), \(\dotsc\), \(u_{\nex}\) dropped before,
their performance is naturally better than that of randomly chosen vectors.
And `Powered \(\Xdrop\)', whose expanded vectors are better approximations
to \(u_{\nes+1}\), \(\dotsc\), \(u_{\nex}\), naturally performs better than
`\(\Xdrop\)'.}
However, the additional cost associated with forming \(A^{-1}\Xdrop\) is the
price to pay.
Therefore, in practice `\(\Xdrop\)' appears to be the most efficient choice.

\section{Conclusions and outlook}
\ISALL
In this work, we propose an acceleration technique that can be applied to
various symmetric block eigensolvers---the shrink-and-expand technique.
Using SI, SD, LOBPCG, and TraceMIN as examples, we worked out all the
implementation details including the position and the timing of applying
the technique.
Both theoretical analysis and numerical experiments support the effectiveness
of the technique.
An acceleration of \(20\%\)--\(30\%\) is observed for most test examples,
demonstrating the potential of the shrink-and-expand technique.
Moreover, the consistency between the observed numerical results and the
theoretical prediction validates our analysis.

The key idea of the shrink-and-expand technique lies in the observation that
the convergence rate of a block eigensolver does not decline immediately after
reducing the block size.
Therefore, the eigensolver stays with nearly the same convergence rate but at
a lower cost.

This work represents a preliminary investigation into the shrink-and-expand
technique.
There is considerable scope for further development of the potential of the
technique.

Firstly, the technique can be applied to a wider range of problems.
In addition to the fixed block size eigensolvers, the shrink-and-expand
technique also has the potential to accelerate block Krylov-based methods.
\RMIS{}{In fact, we have already applied the shrink-and-expand technique to
block Lanczos~\cite{GU1977} and Chebyshev--Davidson~\cite{ZS2007,zhou2010}
algorithms, and also observed an acceleration of approximately \(20\%\).}
Though numerical experiments in this work focus mainly on the smallest magnitude
eigenpairs of Hermitian matrices, the shrink-and-expand
technique naturally carries over to other portion of the spectrum, and
hopefully even for some non-Hermitian matrices.

Secondly, the timing of applying shrinkage and expansion is worth a further
investigation.
For instance, note that the shrink-and-expand technique is not involved until
a threshold \(r^{(j)}\le\rwarm=10^{-4}\) is reached in Section~\ref{sec:numexp}.
However, the threshold here is compared with the residual of the \(\nev\)th
eigenpair.
This is a relatively conservative choice.
Because when the residual of \(\nev\)th eigenpair reaches \(10^{-4}\), other
eigenpairs already have much smaller residuals or even have converged.
In other words, the shrink-and-expand technique we applied only affected a
small amount of eigenpairs that have not converged yet, which may limit the
performance of this technique.
A more aggressive strategy would be to compare the threshold with the residual
of first eigenpair.
Furthermore, for most of test examples, employing a more aggressive
warming up threshold \(\rwarm=10^{-1}\) is also stable, while the acceleration
is, of course, better.

Finally, as demonstrated in both our theoretical analysis and numerical
experiments, the choice of vectors introduced during the expansion
significantly impacts the effectiveness of our technique.
One of the intriguing ideas is to employ mixed-precision arithmetic,
for instance, applying the powered \(\Xdrop\) while using reduced precision
arithmetic for updating \(A^{-1}\Xdrop\).
\ISALLEND

\section*{Acknowledgments}
We would like to thank Jingyu Liu and Jose E. Roman for constructive suggestions.
\RMIS{}{We are grateful to Ding Lu for his comments on the type of expansion
vectors, which made us aware of their significance and, therefore, helped us
refine our proof.
Furthermore, we would like to express our appreciate to the reviewers.
Their feedback helped us improve this work.}

Yuqi Liu and Meiyue Shao are partly supported by the National Natural Science Foundation
of China under grant No.~92370105.
Yuxin Ma is supported by the European Union (ERC, inEXASCALE, 101075632). Views and opinions expressed are those of the authors only and do not necessarily reflect those of the European Union or the European Research Council. Neither the European Union nor the granting authority can be held responsible for them.

\appendix
\ISALL
\section{Proofs of Theorem~\ref{thm:rate}}
\label{sec:proof}

\begin{proof}[Proof of Theorem~\ref{thm:rate}]
By the definition of the principal angle, we have
\[
\tan\angle(V_k,X)
=\left\lVert\bmat{X_{\ell\backslash k} \\
X_\ell^{\perp}}\Xone^{-1}\right\rVert_2,\qquad
\tan\angle(V_k,BX)
=\left\lVert\bmat{\Sigma_{\ell\backslash k}X_{\ell\backslash k}\\
\Sigma_\ell^{\perp}X_\ell^{\perp}}
\Xone^{-1}\Sigma_k^{-1} \right\rVert_2.
\]
Notice that
\begin{equation}
\label{eq:priangle}
\begin{aligned}
&\left\lVert\bmat{\Sigma_{\ell\backslash k}X_{\ell\backslash k} \\
\Sigma_\ell^{\perp}X_\ell^{\perp}}
\Xone^{-1}\Sigma_k^{-1}\right\rVert_2\\
&\qquad\le\left\lVert\bmat{\sigma_{\ell+1}X_{\ell\backslash k} \\
\Sigma_\ell^{\perp}X_\ell^{\perp}}
\Xone^{-1}\Sigma_k^{-1}\right\rVert_2
+\left\lVert\bmat{(\Sigma_{\ell\backslash k}-\sigma_{\ell+1}
I_{\ell-k})X_{\ell\backslash k} \\ 0}
\Xone^{-1}\Sigma_k^{-1}\right\rVert_2\\
&\qquad\le\frac{\sigma_{\ell+1}}{\sigma_k}\tan\angle(V_k,X)
+\frac{\sigma_{k+1}-\sigma_{\ell+1}}{\sigma_k}
\bigl\lVert X_{\ell\backslash k}
\Xone^{-1}\bigr\rVert_2.
\end{aligned}
\end{equation}
Dividing both sides of~\eqref{eq:priangle} by \(\tan\angle(V_k,X)\),
we obtain
\[
\frac{\tan\angle(V_k,BX)}
{\tan\angle(V_k,X)}
\leq\frac{\sigma_{\ell+1}}{\sigma_k}
+\frac{\sigma_{k+1}-\sigma_{\ell+1}}{\sigma_k}
\frac{\bigl\lVert X_{\ell\backslash k}
\Xone^{-1}\bigr\rVert_2}
{\left\lVert\bmat{X_{\ell\backslash k} \\
X_\ell^{\perp}}\Xone^{-1}\right\rVert_2}.
\qedhere
\]
\end{proof}

\section{Proofs of Theorem~\ref{thm:decomp}}
\label{sec:proof-decomp}
\begin{proof}[Proof of Theorem~\ref{thm:decomp}]
To obtain \(Y\), we aim to orthogonalize \(X\Xone^{-1}\) with respect to
\(\Xexp\).
This process involves two main steps: first eliminate the component of
\(\Xexp\) from \(X\Xone^{-1}\), and then orthogonalize the resulting matrix.

We first remove the component of \(\Xexp\) from \(X\Xone^{-1}\) by
\[
\begin{aligned}
\Xmid&=X\Xone^{-1}-\Xexp(\Xexp\herm X\Xone^{-1})\\
&=V\left(\bmat{I_k\\\Xtwo\Xone^{-1}\\\Xthr\Xone^{-1}}
-\bmat{\epsexp\cdot\Deltaone\\
I_{\ell-k}+\epsexp\cdot\Deltatwo\\\epsexp\cdot\Deltathr}
(\Xexp\herm X\Xone^{-1})\right)\\
&=V\bmat{I_k-\epsexp\cdot\Deltaone(\Xexp\herm X\Xone^{-1})\\
-\epsexp\cdot\Delta\herm X\Xone^{-1}
-\epsexp\cdot\Deltatwo(\Xexp\herm X\Xone^{-1})\\
\Xthr\Xone^{-1}-\epsexp\cdot\Deltathr(\Xexp\herm X\Xone^{-1})}.
\end{aligned}
\]
To ensure orthogonality, we multiply \(\Xmid\) by
\((\Xmid\herm\Xmid)^{-1/2}\), leading to the final expression for \(Y\) as
\[
Y=\bigl[\Xmid(\Xmid\herm\Xmid)^{-1/2}, \Xexp\bigr].
\]
It is not difficult to verify that when \(\Xmid\) is of full column rank,
\(Y\) is orthogonal and \(\Span\{Y\}=\Span\{X,\Xexp\}\).

We then introduce an auxiliary term \(F=(\Xmid\herm\Xmid)^{-1/2}-I_l\) to
reformulate \(Y\) in a more clear form
\[
Y=V_\ell+V\bmat{E_{1,1} & E_{1,2}\\
E_{2,1} & E_{2,2}\\
E_{3,1} & E_{3,2}},
\]
where the \(E_{*,*}\) blocks are given by
\[
\begin{aligned}
E_{1,1}&=-\epsexp\cdot\Deltaone(\Xexp\herm X\Xone^{-1})(I_l+F)+F,&
E_{1,2}&=\epsexp\cdot\Deltaone,\\
E_{2,1}&=-\bigl(\epsexp\cdot\Delta\herm X\Xone^{-1}
+\epsexp\cdot\Deltatwo(\Xexp\herm X\Xone^{-1})\bigr)(I_l+F), &
E_{2,2}&=\epsexp\cdot\Deltatwo,\\
E_{3,1}&=\bigl(\Xthr\Xone^{-1}-\epsexp\cdot
\Deltathr(\Xexp\herm X\Xone^{-1})\bigr)(I_l+F), &
E_{3,2}&=\epsexp\cdot\Deltathr.\\
\end{aligned}
\]
Remind that \(\lVert\Deltaone\rVert_2\le\lVert\Delta\rVert_2\),
\(\lVert\Deltatwo\rVert_2\le\lVert\Delta\rVert_2\),
\(\lVert\Deltathr\rVert_2\le\lVert\Delta\rVert_2\), and
\(\lVert\Delta\rVert_2=\lVert\Xexp\rVert_2=1\), it is not hard to obtain
\[
\begin{aligned}
\lVert E_{1,1}\rVert_2&\leq\epsexp\heta(1+\lVert F\rVert_2)+\lVert F\rVert_2,&
\lVert E_{1,2}\rVert_2&\leq\epsexp,\\
\lVert E_{2,1}\rVert_2&\leq2\epsexp\heta(1+\lVert F\rVert_2), &
\lVert E_{2,2}\rVert_2&\leq\epsexp,\\
\lVert E_{3,1}\rVert_2&\leq\teta+\teta\lVert F\rVert_2
+\epsexp\heta(1+\lVert F\rVert_2), &
\lVert E_{3,2}\rVert_2&\leq\epsexp,\\
\lVert E_{3,1}\rVert_2&\geq\teta-\teta\lVert F\rVert_2
-\epsexp\heta(1+\lVert F\rVert_2).\\
\end{aligned}
\]

The last thing to be done is to estimate \(\lVert F\rVert_2\).
Notice that
\[
\begin{aligned}
&\Xmid\herm\Xmid-I_l\\
&=-\epsexp\cdot\Deltaone(\Xexp\herm X\Xone^{-1})
-\epsexp\cdot(\Xexp\herm X\Xone^{-1})\herm\Deltaone\herm\\
&\quad+\epsexp^2\cdot\bigl(\Deltaone
(\Xexp\herm X\Xone^{-1})\bigr)\herm\Deltaone(\Xexp\herm X\Xone^{-1})\\
&\quad+\epsexp^2\cdot\bigl(\Delta\herm X\Xone^{-1}
+\Deltatwo(\Xexp\herm X\Xone^{-1})\bigr)\herm
\bigl(\Delta\herm X\Xone^{-1}+\Deltatwo(\Xexp\herm X\Xone^{-1})\bigr)\\
&\quad+(\Xthr\Xone^{-1})\herm\Xthr\Xone^{-1}\\
&\quad-\epsexp\cdot(\Xthr\Xone^{-1})\herm\Deltathr(\Xexp\herm X\Xone^{-1})
-\epsexp\cdot\bigl(\Deltathr
(\Xexp\herm X\Xone^{-1})\bigr)\herm\Xthr\Xone^{-1}\\
&\quad+\epsexp^2\cdot\bigl(\Deltathr
(\Xexp\herm X\Xone^{-1})\bigr)\herm\Deltathr(\Xexp\herm X\Xone^{-1}).\\
\end{aligned}
\]
Thus, when \(\epsexp\) is sufficiently small, we can assume \(\Xmid\herm\Xmid-I_l\)
is positive definite.
Using~\cite[Lemma 1]{LSS2024}, we obtain the bound
\[
\begin{aligned}
2\lVert F\rVert_2
&=2\bigl\lVert(\Xmid\herm\Xmid)^{-1/2}-I_l\bigr\rVert_2\\
&\le\lVert\Xmid\herm\Xmid-I_l\rVert_2\\
&\le\teta^2+2\epsexp(\teta\heta+\heta)+6\epsexp^2\heta^2.
\qedhere
\end{aligned}
\]
\end{proof}

\section{Proofs of Theorem~\ref{thm:main}}
\label{sec:proof-main}
\begin{proof}[Proof of Theorem~\ref{thm:main}]
Using the conclusion of Theorem~\ref{thm:decomp}, we first construct
an orthogonal basis \(Y\) of \(\Span\{X,\Xexp\}\) of the
form~\eqref{eq:fromy}.
Then, as proved in Theorem~\ref{thm:perturbation}, during the
Rayleigh--Ritz process, the first \(k\) columns of \(YC\) will be
selected as \(\Xnew\), where \(C\) is of the form~\eqref{eq:formc}.
This means that \(\Xnew\) can be expressed as
\[
\Xnew=YC\bmat{I_k\\0}=V\bmat{
(I_k+E_{1,1})(C_1+\delta C_{1,1})+E_{1,2}\delta C_{2,1}\\
E_{2,1}(C_1+\delta C_{1,1})+(I_{l-k}+E_{2,2})\delta C_{2,1}\\
E_{3,1}(C_1+\delta C_{1,1})+E_{3,2}\delta C_{2,1}
}
=V\bmat{\Xnew_k\\
\Xnew_{l\backslash k}\\
\Xnew_l^{\perp}},
\]
where, just as~\eqref{eq:formx}, we decompose \(\Xnew\) into three parts:
\(\Xnew_k\), \(\Xnew_{l\backslash k}\), and \(\Xnew_l^{\perp}\).

We begin with estimating the norm of \(\Xnew_k^{-1}\).
Note that
\begin{align*}
\Xnew_k^{-1}
&=\bigl((I_k+E_{1,1})(C_1+\delta C_{1,1})
+E_{1,2}\delta C_{2,1}\bigr)^{-1}\\
&=\bigl(C_1+\delta C_{1,1}+E_{1,1}(C_1+\delta C_{1,1})
+E_{1,2}\delta C_{2,1}\bigr)^{-1}.
\end{align*}
The perturbation term satisfies
\[
\lVert\delta C_{1,1}+E_{1,1}(C_1+\delta C_{1,1})+E_{1,2}\delta C_{2,1}\rVert_2
\le\frac{1}{2}\teta^2+O(\epsexp).
\]
Suppose that this term is sufficiently small so that
\[
\frac{1}{2}\teta^2+O(\epsexp)<\frac{1}{2},
\]
Then \(\bigl\lVert\Xnew_k^{-1}-C_1^{-1}\bigr\rVert_2\) is bounded by
\begin{equation*} \label{eq:normXnewkinv-C1inv}
\begin{aligned}
\bigl\lVert\Xnew_k^{-1}-C_1^{-1}\bigr\rVert_2
&\le\bigl\lVert C_1^{-1}-\bigl(C_1+\delta C_{1,1}+E_{1,1}(C_1+\delta C_{1,1})
+E_{1,2}\delta C_{2,1}\bigr)^{-1}\bigr\rVert_2\\
&\le \frac{\lVert\delta C_{1,1}+E_{1,1}(C_1+\delta C_{1,1})
+E_{1,2}\delta C_{2,1}\rVert_2}
{1-\lVert\delta C_{1,1}+E_{1,1}(C_1+\delta C_{1,1})
+E_{1,2}\delta C_{2,1}\rVert_2}\\
&\le 2\lVert\delta C_{1,1}+E_{1,1}(C_1+\delta C_{1,1})
+E_{1,2}\delta C_{2,1}\rVert_2\\
&=\teta^2+O(\epsexp),
\end{aligned}
\end{equation*}
where the second inequality is from~\cite[Equation (5.8.4)]{HJ2012}.
It follows that
\[
    \bigl\lVert\Xnew_k^{-1}\bigr\rVert_2
    \leq \bigl\lVert C_1^{-1}\bigr\rVert_2 + \bigl\lVert\Xnew_k^{-1}-C_1^{-1}\bigr\rVert_2
    \leq 1 + \teta^2+O(\epsexp).
\]

With this estimate of \(\bigl\lVert\Xnew_k^{-1}\bigr\rVert_2\), we can now
bound~\eqref{eq:newterm}.
First, we estimate its numerator \(\Xnew_{l\backslash k}\Xnew_k^{-1}\) as
\begin{equation}
\label{eq:estxtwo}
\begin{aligned}
\bigl\lVert\Xnew_{l\backslash k}\Xnew_k^{-1}\bigr\rVert_2
&\le\bigl\lVert\Xnew_{l\backslash k}\bigr\rVert_2
\bigl\lVert\Xnew_k^{-1}\bigr\rVert_2\\
&=\lVert E_{2,1}(C_1+\delta C_{1,1})+(I+E_{2,2})\delta C_{2,1}\rVert_2
\bigl\lVert\Xnew_k^{-1}\bigr\rVert_2\\
&=O(\epsexp).
\end{aligned}
\end{equation}
Next, we analyze the denominator of~\eqref{eq:newterm} as
\begin{equation}
\label{eq:estxthr1}
\begin{aligned}
\quad\left\lVert\bmat{\Xnew_{ l\backslash k}\\
\Xnew_l^{\perp}}\Xnew_k^{-1}\right\rVert_2
&\ge\bigl\lVert\Xnew_l^{\perp}\Xnew_k^{-1}\bigr\rVert_2\\
&\ge\bigl\lVert\Xnew_l^{\perp}\bigl(C_1^{-1}
+(\Xnew_k^{-1}-C_1^{-1})\bigr)\bigr\rVert_2\\
&\ge\bigl\lVert\Xnew_l^{\perp}C_1^{-1}\bigr\rVert_2
-\bigl\lVert\Xnew_l^{\perp}(\Xnew_k^{-1}-C_1^{-1})\bigr\rVert_2\\
&\ge\bigl\lVert\Xnew_l^{\perp}\bigr\rVert_2
-\bigl\lVert\Xnew_l^{\perp}\bigr\rVert_2\bigl\lVert\Xnew_k^{-1}-C_1^{-1}\bigr\rVert_2\\
&\ge\bigl\lVert\Xnew_l^{\perp}\bigr\rVert_2\bigl(1-\teta^2-O(\epsexp)\bigr)\\
&\ge\teta-\frac{3}{2}\teta^3+\frac{1}{2}\teta^5-O(\epsexp),
\end{aligned}
\end{equation}
where the fourth inequality is base on the fact that \(C_1^{-1}\) is unitary,
and in the last line we used the bound
\begin{equation}
\label{eq:estxthr2}
\begin{aligned}
\bigl\lVert\Xnew_l^{\perp}\bigr\rVert_2
&=\lVert E_{3,1}C_1+E_{3,1}\delta C_{1,1}+E_{3,2}\delta C_{2,1}\rVert_2\\
&\ge\lVert E_{3,1}C_1\rVert_2-\lVert E_{3,1}\delta C_{1,1}+E_{3,2}\delta C_{2,1}\rVert_2\\
&\ge\teta-\frac{1}{2}\teta^3-O(\epsexp).
\end{aligned}
\end{equation}
Combining~\eqref{eq:estxtwo}, \eqref{eq:estxthr1} and~\eqref{eq:estxthr2},
we arrive at
\begin{equation}
\label{eq:convterm_xnew}
\frac{\bigl\lVert\Xnew_{ l\backslash k}\Xnew_k^{-1}\bigr\rVert_2}
{\left\lVert\bmat{\Xnew_{ l\backslash k}\\
\Xnew_l^{\perp}}\Xnew_k^{-1}\right\rVert_2}
\le\frac{O(\epsexp)}
{\teta-\frac{3}{2}\teta^3+\frac{1}{2}\teta^5-O(\epsexp)}.
\end{equation}
Substituting~\eqref{eq:convterm_xnew} into~\eqref{eq:convergerate}, we
obtain the conclusion.
\end{proof}

\ISALLEND

\end{document}